\documentclass[]{amsart}
\usepackage[english]{babel}

\usepackage{epsfig,amsmath}
\usepackage{xspace}
\usepackage[psamsfonts]{amssymb}
\usepackage[latin1]{inputenc}
\usepackage{float}\restylefloat{figure}
\usepackage{color}

\newtheorem*{lemma*}{\bf Lemma}
\newtheorem*{sublemma*}{\bf Sublemma}
\newtheorem*{claim*}{\bf Claim}
\newtheorem*{complement*}{\bf Complement}

\renewcommand{\epsilon}{\varepsilon}

\newcommand{\cal}[1]{\mathcal{#1}}
\newcommand{\wt}[1]{\widetilde{#1}}
\newcommand{\setof}[2]{\big\{{#1}\,\big|\,{#2}\big\}}

\newcounter{rememberItem}

\def\bEA{\begin{eqnarray*}}
\def\eEA{\end{eqnarray*}}
\def\bEAn{\begin{eqnarray}}
\def\eEAn{\end{eqnarray}}

\def\ds{\displaystyle}
\def\tend{\longrightarrow}

\def\ds{\displaystyle}
\def\wh{\widehat}
\def\on{\operatorname}

\def\cal{\mathcal}

\def\C{{\mathbb C}}
\def\D{{\mathbb D}}

\def\N{{\mathbb N}}

\def\remark{\vskip.2cm \noindent{\bf Remark. }}
\def\endremark{\par\medskip}

\graphicspath{{img/}}
\usepackage{amscd}
\usepackage{enumerate}
\usepackage[pdftex]{hyperref}
\usepackage{tikz}
\usepackage{array} 
\usetikzlibrary{arrows}
\definecolor{dkgreen}{rgb}{0,0.5,0}

\newtheorem{lemma}{Lemma}

\newtheorem{theorem}{Theorem}

\newtheorem*{theorem*}{Theorem}

\newtheorem{question}{Question}

\newcommand{\ov}{\overline}

\newcommand{\AutC}{\on{Aut}(\wh\C)}

\begin{document}

\title[Tan Lei and Shishikura's example]{Tan Lei and Shishikura's example of non-mateable degree 3 polynomials without a Levy cycle}
\author{Arnaud Chéritat}
\address{Centre National de la Recherche Scientifique\\
Institut de Mathématiques de Toulouse\\
Université Paul Sabatier
118 Route de Narbonne
31062 Toulouse Cedex 9,
France}
\email{arnaud.cheritat@math.univ-toulouse.fr}
\thanks{This research was funded by the grant ANR--08--JCJC--0002 of the Agence Nationale de la Recherche.}

\abstract After giving an introduction to the procedure dubbed \emph{slow polynomial mating} and stating a conjecture relating this to other notions of polynomial mating, we show conformally correct pictures of the slow mating of two degree $3$ post critically finite polynomials introduced by Shishikura and Tan Lei as an example of a non matable pair of polynomials without a Levy cycle. The pictures show a limit for the Julia sets, which seems to be related to the Julia set of a degree $6$ rational map. We give a conjectural interpretation of this in terms of pinched spheres and show further conformal representations.
\endabstract

\maketitle

\tableofcontents

\section{Introduction} 

In \cite{TS}, Tan Lei and Shishikura gave an example of two post critically finite polynomials of degree $3$ whose ``formal'' mating has an obstruction which is not a levy cycle.

Let us recall the context:

\subsection{Context}

One thing making the discussion difficult, but also interesting, is that there are several non-equivalent definitions of polynomial matings.
Let us recall the definition of the slow variant (see \cite{M} p.~54):

\subsubsection{Definition of slow polynomial mating}

\begin{figure}%
\begin{tikzpicture}
\node at (0,0) {\includegraphics[width=11.5cm]{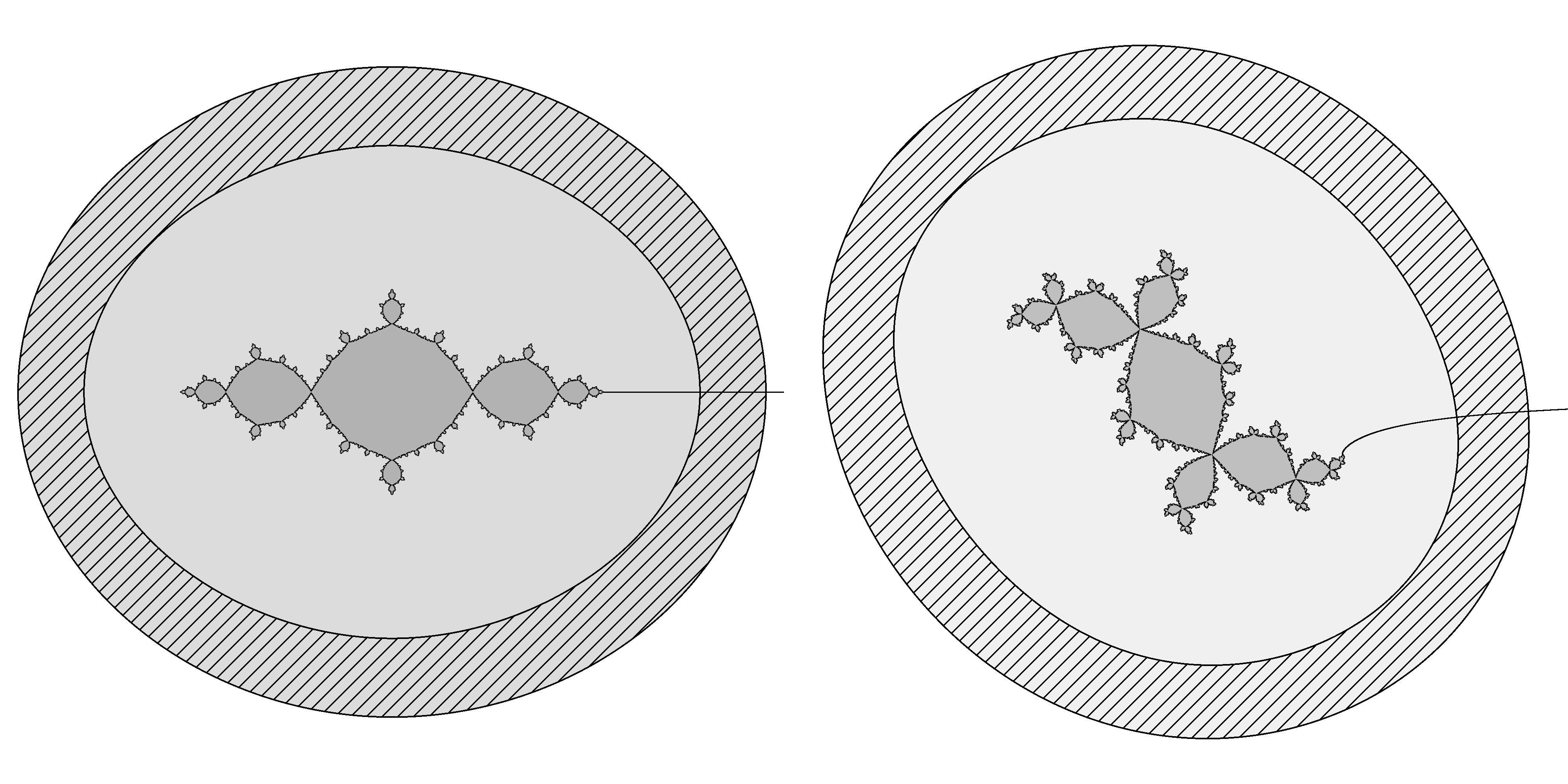}};
\draw (-5.45,-2.1) node {$U_1$};
\draw (0.75,-2.1) node {$U_2$};
\node at (0,-6.5) {\includegraphics[height=6cm]{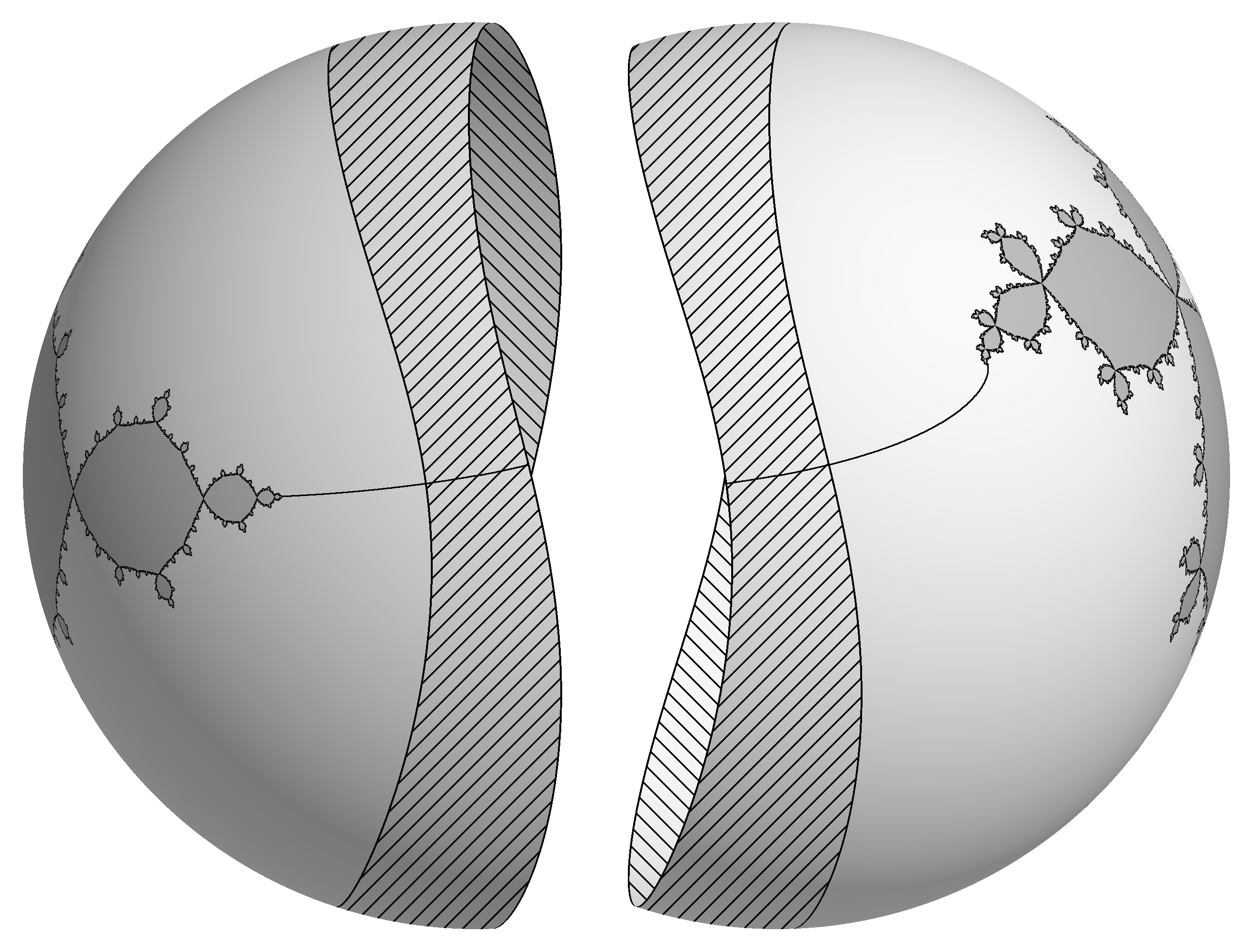}};
\end{tikzpicture}
\caption{Example of construction of the Riemann surface $\cal S_R$. The images depict the Julia sets nicknamed the Basilica and the Rabbit. Some value of $R$ has been chosen and the domain bounded by the equipotential $P'>\log R$ is indicated by shades of gray. The respective external rays of argument $0$ have also been drawn. The parts of $U_i$ that get glued together are hatched. In the image below we presented the two charts facing each other in a 3D space so that points glued together are close.}
\label{fig:sr1}
\end{figure}

\begin{figure}%
\begin{tikzpicture}[>=latex]
\node at (0,0) {\includegraphics[width=12cm]{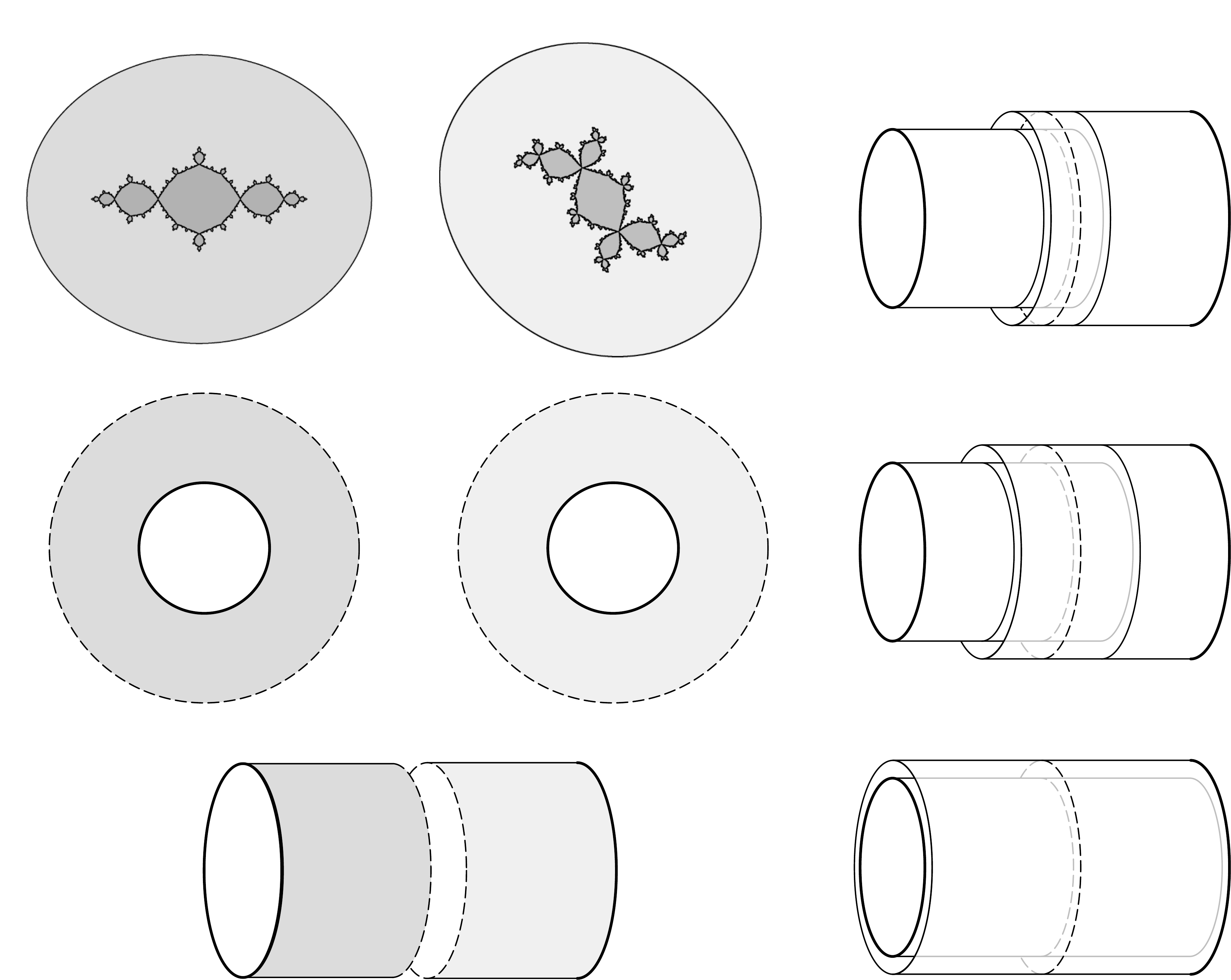}};
\draw[->] (-4,1.7)  -- node[left] {$\phi_1$} (-4,0.5);
\draw[->] (0,1.7)  -- node[left] {$\phi_2$} (0,0.5);
\draw[->] (-3.4,-1.6)  -- node[right] {$\log$} (-2.7,-3);
\draw[->] (-0.3,-1.6)  -- node[left] {$-\log$} (-1,-3);
\draw (-2.5,1.4) node[anchor=north west,inner sep=0mm] {$U_1(R)$} -- (-2.9,1.8);
\draw (1.55,1.4) node[anchor=north west,inner sep=0mm] {$U_2(R)$} -- (1.15,1.8);
\end{tikzpicture}%
\caption{Left: gluing $U_1(R)$ and $U_2(R)$ is gluing their Böttcher coordinates along the circle $|z|=R$ by an inversion. Better, it is gluing the logarithms of the Böttcher coordinates with an isometry. Right: in fact to define a Riemann surface we take slightly larger domains so that open sets are glued together. This depends on another parameter $P'$ (with $P<P'\leq 2P$) which determines the overlap. However the resulting surfaces are the same (canonically isomorphic).
}
\label{fig:sr_eq}
\end{figure}
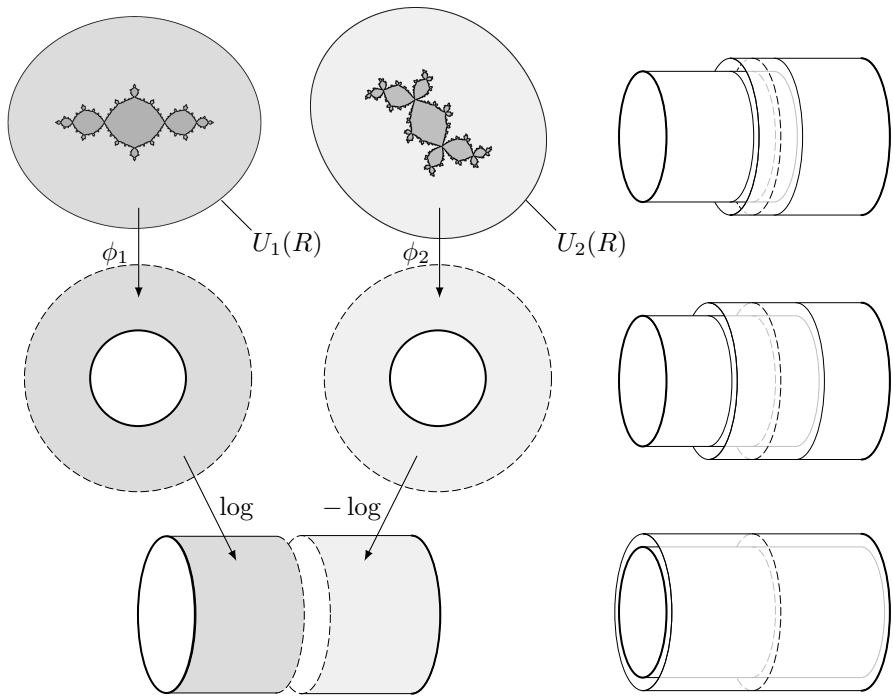

Consider two polynomials $P_1$, $P_2$ of the same degree $d\geq 2$. For the construction to make sense we need that their Julia sets be connected. 

\remark No other assumption is done. In  other constructions it is either required that the Julia sets are locally connected or even stronger: that the polynomials are post-critically finite. 
\endremark

Recall that a polynomial of degree $d\geq2$ with connected Julia set has $d-1$ fixed external rays (i.e.\ of period $1$). The construction that follows depends on the choice of such an external ray for $P_1$, and of one for $P_2$. In the degree $2$ case, since there is only one fixed external ray, there is no choice.

For any real number $P>0$ let $R=e^{P}$ and let $U_i(R)$ be the simply connected open subset of $\C$ bounded by the equipotential of $P_i$ with potential $P$.
Bascially we would like to build a Riemann surface $\cal S_R$ by gluing the closure of $U_1(R)$ and $U_2(R)$ along their boundaries in some specific fashion. As long as the gluing is holomorphic, it is possible to get a Riemann surface this way. However we prefer to present the construction slightly differently, by gluing along open regions.

\begin{figure}%
\begin{tikzpicture}
\node at (0,0) {\includegraphics[width=11.5cm]{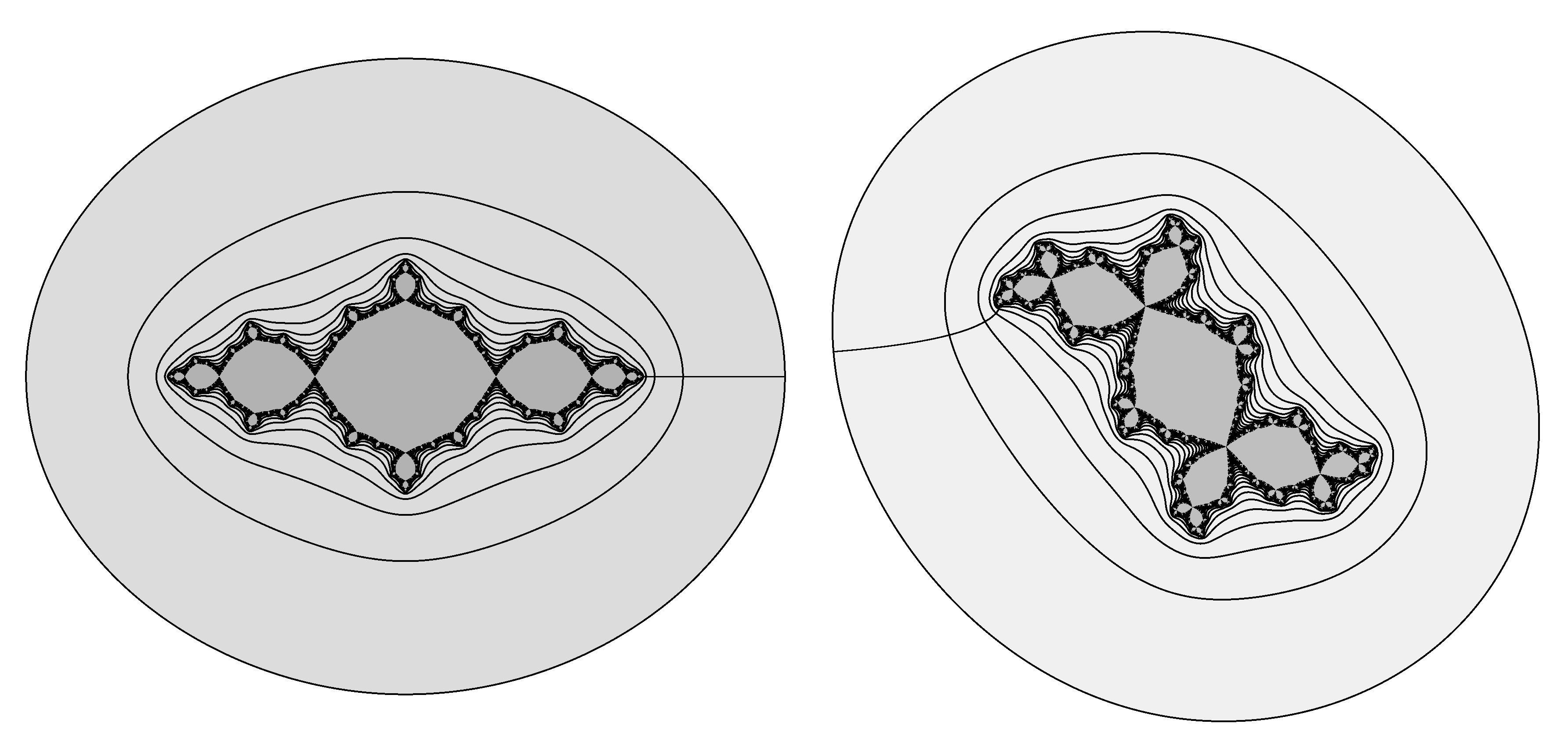}};
\end{tikzpicture}
\caption{Here we drew the domains delimited by the potential $P$ together with more equipotentials. The image on the right has been rotated 180 degrees so that now the external rays of argument $0$ are facing each other. On $\cal S_R$, the boundaries of these two domains are glued together and the point on the external ray $0$ match. On Figure~\ref{fig:sr3}, we see what these pictures become after being glued to form $\cal S_R$.}
\label{fig:sr2}
\end{figure}

So choose $P'\in\ ]P,2P]$, and let 
\[U_i = U_i(R').\]
Now let $\phi_1$ be the unique Böttcher coordinate for $P_1$ that maps the chosen fixed external ray to the interval $]\,1,+\infty\,[$,
let $\phi_2$ be the same for $P_2$ and glue $z\in U_1$ with $w\in U_2$ whenever
\[\frac{\phi_1(z)}{R}\cdot\frac{\phi_2(w)}{R}=1.\]
Note that this relation restricts to an orientation reversing bijection between the equipotentials of level $P$ of $P_1$ and $P_2$. Note also that if $z\sim w$ then their external angles\footnote{arguments of their Böttcher coordinates} are opposite. The gluing can equivalently be expressed in the coordinates $\log \phi_i$: 
\[\frac{\log \phi_1(z) + \log\phi_2(z)}{2} = P \pmod{2\pi i}.\]
It can be checked that the quotient topological space under this equivalence relation is indeed a topological manifold. Moreover, $U_1$ and $U_2$ provide an atlas for which the change of coordinates is holomorphic and thus we have defined a Riemann surface.
It does not depend on the choice of $P'\in\ ]P,2P]$ in the sense that there are canonical isomorphisms between two such surfaces (see Figure~\ref{fig:sr_eq}) and will be designated as $\cal S_R$ regardless of the value of $P'$.

\begin{figure}%
\includegraphics[width=10cm]{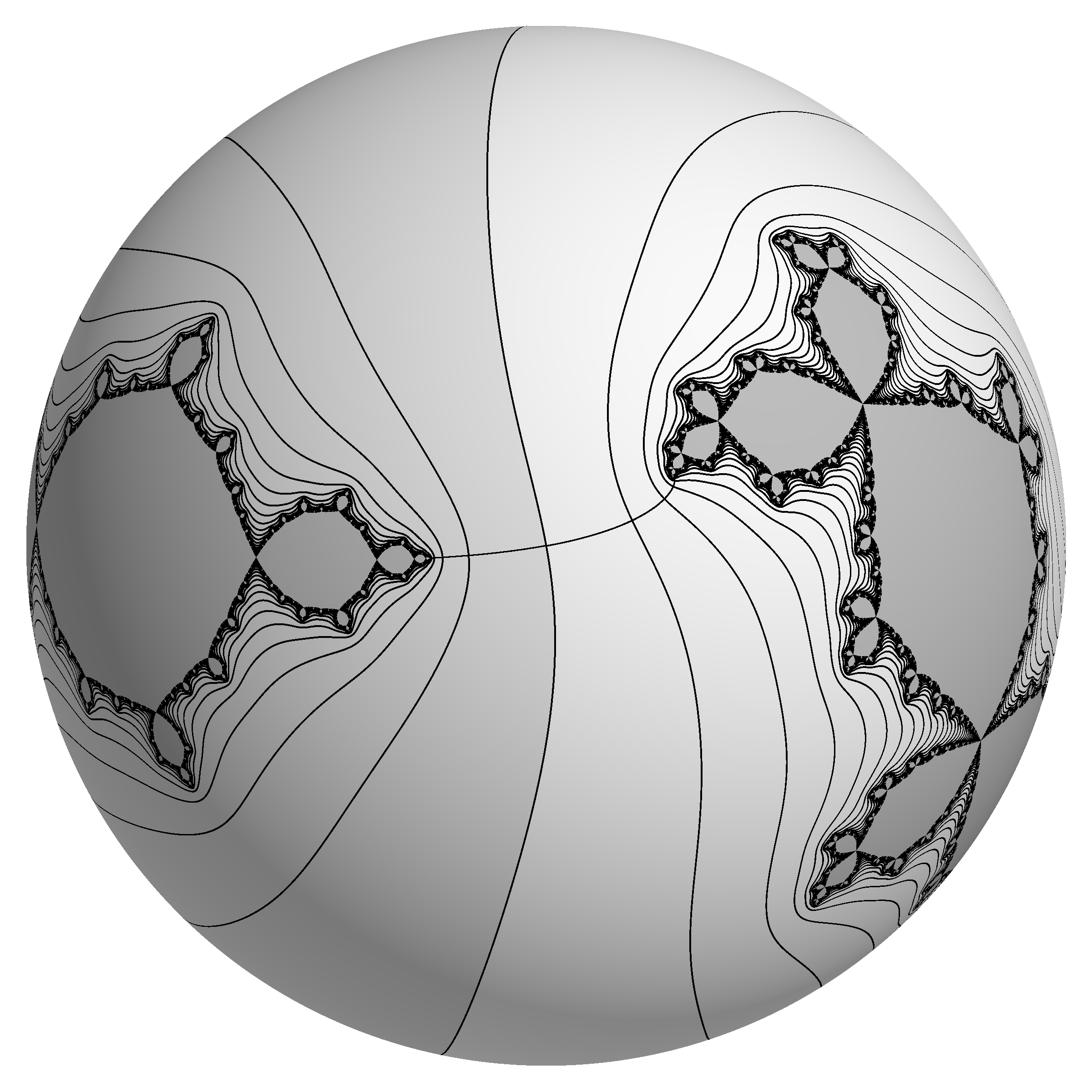}
\caption{The Riemann surface $\cal S_R$ conformally mapped to a Euclidean sphere, painted with the drawings of Figure~\ref{fig:sr2}. The method for producing such a picture is interesting and explained in another article in the present volume: it does not work by computing the conformal map, but instead by pulling-back Julia sets by a series of rational maps. It has connections with Thurston's algorithm.}
\label{fig:sr3}
\end{figure}

All these surfaces are homeomorphic to the sphere, so they are isomorphic to the Riemann sphere $\wh\C$. Recall that the latter has a group of automorphisms which is not small\footnote{It is the group of homographies a.k.a.\ Möbius transformations or projective self maps, i.e.\ those maps of the form $z\mapsto (az+b)/(cz+d)$ that are not constant. The group is also sharply $3$-transitive: given a triple of pairwise distinct points $a_1$, $a_2$, $a_3$ in $\wh\C$ an another one $b_1$, $b_2$, $b_3$ there is a unique Möbius map sending each $a_i$ to $b_i$.} so the isomorphism to $\wh\C$ is not unique. Its choice will play a determinant role when we look at degenerating cases. The choice of such an isomorphism is called a normalization.

There are several objects of interest than can be defined on $\cal S_R$. First the Julia sets of $P_1$ and $P_2$ correspond to two compact subsets of $\cal S_R$. The same holds for the filled-in Julia sets. Equipotentials and external rays are also well-defined, together with a conformal map $\Phi$ from $A$ to the round annulus $R^{-1}<|z|<R$ where $A$ is the the complement of the union of the two filled-in Julia sets:
\[\Phi: A \to \setof{z}{R^{-1}<|z|<R}.\]
Let us designate these subsets of $\cal S_R$ by $J_1(R)$, $K_1(R)$, $J_2(R)$, $K_2(R)$ and $A(R)$. The equator of $A(R)$ is the preimage of the unit circle by $\Phi$ and corresponds in either chart $U_i$ to the equipotential $\log R$.
More complicated objects can be defined. For instance the measure $\mu_R$ on the equator of $A(R)$ which is the pull-back by $\Phi$ of the uniform probability measure on the unit circle.

There is an important subtlety about the transfer to $\cal S_R$ of the dynamical system. It turns out that $P_1$ together with $P_2$ define a holomorphic map from $\cal S_R$ to $\cal S_{R^d}$, \emph{not} to $\cal S_R$:
\[\cal S_R\overset{F_R}{\longrightarrow}\cal S_{R^d}.\]
To explain this,
let us recall that $R=e^P$ and $P'\in]P,2P]$. Let $R'=e^{P'}$. Then $P_i$ maps $U_i(R')$ to $U_i(R'^d)$ and if $z\in U_i\setminus K(P_i)$ then the Böttcher coordinate satisfies $\phi_i(P_i(z))=\phi_i(z)^d$. Hence if the product $\phi_1(z)\phi_2(w)=R^{2}$ then $\phi_1(P_1(z))\phi_2(P_2(z))=R^{2d}$.
\footnote{In the degree $2$ case it would be tempting to try to define a map from $\cal S_R$ to itself by deciding to map each half-annulus to the whole annulus. However it would be discontinuous at the equator.}
If isomorphisms from $\cal S_R$ and $\cal S_{R^d}$ to $\wh \C$ have been chosen then the holomorphic map $F_R$ becomes a \emph{rational map of degree d}.

Hence we get a one real parameter family of rational maps $F_R$ of degree $d$ between spheres with distinguished compact subsets. Note that $F_R(K_i(R))=K_i(R^d)$, and the same holds for $J_i(R)$ and $A(R)$. Equipotentials and external rays are mapped to equipotentials and external rays.

\subsubsection{Limits of slow mating?}
Now, as $R\tend 1$ by values $>1$, do these objects have limits? Are these limits compatible? It depends on the initial maps $P_1$ and $P_2$, but also on the meaning we give to having a limit and to compatibility.
One thing we could ask for is to find a normalization of the spheres $\cal S_R$ such that the maps $F_R$ tend to a rational map $F$ of the same degree (which is $d$). We would then like both sets $J_i(R)$ to tend to the Julia set of $F$. We could ask more, for instance\footnote{The author ignores if this particular requirement is implied by the previous ones.} that the equatorial measure $\mu_R$ tends to the Brolin-Lyubich measure of $F$. We could also ask for less, for instance that $F_{R_n}$ converges for some sequence $R_n\tend 1$, or even allow $F_R$ to tend to a lower degree map (convergence occuring on the complement of a finite set). The present article does not study these questions, it just shows an example that degenerates in a specific way.

By \emph{slow mating} we mean the definition of the Riemann surfaces $\cal S_R$ and the objects $J_i(R)$, \ldots, $F_R$ that we defined on or between them. We do not designate the (hypothetical) limit map $F$.
Let us stress out once more that the slow mating of two polynomials of the same degree $d>1$ is well-defined whenever both have connected Julia sets. They do not need to be post-critically finite, and the Julia sets do not even need to be locally connected. The question of the convergence of $F_R$ can be posed in the non-locally connected case and is non-trivial. If there is a limit, it is not anymore possible to link it to the \emph{topological mating},\footnote{it may also be called the \emph{instant mating} by opposition to the slow mating} obtained by gluing $K_1$ and $K_2$ along their boundaries respecting the Caratheodory loop (see \cite{R}, \cite{S}, \cite{TS}, \cite{M2}), because there is no more Caratheodory loop!

Remark: A twist can be added to the gluing: see \cite{BEK}. The real parameter $R>1$ then becomes a complex parameter $\rho = Re^{i\theta}\in\C\setminus\ov{\D}$ and we get Riemann surfaces $\cal S_\rho$, objects $J_i(\rho)$, etc\ldots and maps $F_\rho:\cal S_\rho \to\cal S_{\rho^d}$. We will not study this here.

Remark: Traditionally the mating of post-critically finite polynomials is not defined in terms of the slow mating but either with the topological mating (PCF polynomials have locally connected Julia sets), see for instance \cite{R} and \cite{S}, or in terms of corrected formal matings, see \cite{TS}. Relations between these notions have been studied but not completely yet.

\subsection{Post-critically finite polynomials}

From now on, with a few explicitly indicated exceptions, we focus on the case where $P_1$ and $P_2$ are post-critically finite, abbreviated as PCF throughout this article.

\subsubsection{About the Thurston machinery}

We recall that a PCF orientation preserving ramified self cover (continuous but not assumed holomorphic) of the (oriented) sphere is called a \emph{Thurston map}.
We assume the reader knows the definition of Thurston equivalence of Thurston maps, the definition of a Thurston map with hyperbolic orbifold, the definition of Thurston obstructions and the related existence and uniqueness theorem of rational realizations of classes of Thurston maps with hyperbolic orbifold. We do not assume the reader knows the proof, but at some point we will use Thurston's pull-back map (a.k.a.\ Thurston's algorithm). These subjects are discussed in details in \cite{DH}. Note that in the present article, we \emph{do not} require an obstruction to be a \emph{stable} multicurve.\footnote{Also called invariant multicurve. A multicurve $\Gamma$ for $f$ is called stable if all components of $f^{-1}(\Gamma)$ are either peripheral or homotopic rel.\ the post-critical set to a curve in $\Gamma$.}

\subsubsection{The marked set}

In the case where $P_1$ and $P_2$ are post-critically finite, one may map the post-critical set\footnote{the union of critical values and their orbits} of each polynomials to $\cal S_R$ (let us call $\cal P_R\subset\cal S_R$ this finite collection of points) and try to see if these points remain at some distance from each other as $R\tend 1$ of if they tend to gather in different groups. It is important to stress out that $\cal P_R$ is not\footnote{unless accidentally} the post-critical set\footnote{However, it is the union over $n>0$ of the critical values of the compositions $F_{R^{1/d}}\circ F_{R^{1/d^2}} \circ \cdots \circ F_{R^{1/d^n}}$} of some $F_{R'}$. This is why we prefer to call $\cal P_R$ the \emph{marked set}.
One may also try to guess what is the Thurston equivalence class of the limit map $F$ of the sequence $F_R$ if there is any. A candidate is provided by the so-called \emph{formal mating}, which is a Thurston map built from $P_1$ and $P_2$ (see \cite{R}, \cite{S}, \cite{TS} and \cite{M2}; we also give a description in Section~\ref{subsub:constr} of the present article). In this article, it will be called the \emph{candidate}.
However there are some cases where the map $F_R$ has a post-critically finite limit of degree $d$ but of a different Thurston class: this is not because some points in $\cal P_R$ collapse at the limit, that the degree necessarily drops. This case has been acknowledged and a way to correct the formal mating is known (see Section~\ref{subsec:topovsform}), though the author does not know if the link between corrected formal mating and limit of slow mating has been completely proved to hold.

\subsubsection{Relation between slow mating and Thurston's algorithm}\label{subsub:rel}

Recall that the set $\AutC$ of automorphisms of the Riemann sphere $\wh \C$ consists of Möbius maps.
Let $f$ be the formal mating of $P_1$ and $P_2$. It is a Thurston map defined on some oriented topological sphere $S$. In Thurston's algorithm, this sphere and the post-critical set $P_f\subset S$, serve as a marker for a \emph{Teichmüller space}: $\cal T$. We remind the reader that $\cal T$ is the set of equivalence classes, denoted $[\phi]$, of orientation preserving homeomorphisms $\phi: S\to \wh \C$ for the relation $\phi\sim\phi'$ iff $\exists \mu\in\AutC$ such that $\mu \circ \phi$ coincides with $\phi'$ on $P_f$ and is isotopic rel $P_g$ to $\phi'$ on the rest of $S$. 
The associated Thurston pullback map will be denoted $\sigma_f:\cal T\to\cal T$. Recall that for $f$ PCF with hyperbolic orbifold, $\sigma_f$ is weakly contracting\footnote{distances are not increasing} and $\sigma_f^2$ is locally strictly contracting, for some metric on $\cal T$ called the Teichmüller metric. Also, $f$ is realizable by a rational map if and only if $\sigma_f$ has a fixed point. Therefore, $f$ is realizable if and only if $\sigma_f^n([\phi])$ has a limit in $\cal T$, and this is independent of the starting point $[\phi]$.

Now, the Riemann surfaces $\cal S_R$ defined by the slow mating come with natural\footnote{unique up to isotopy rel $P_f$} markings by $S$, i.e\ maps \[\phi_R: S\to \cal S_R.\]
To specify such a marking we need to specify a construction of $S$ and $f$ first, and there is some flexibility in the choice. Even with a given choice, there is still some flexibility in the definition of $\phi_R$. In Section~\ref{subsub:constr}, we give one example of construction with the following nice property:
\[F_R\circ\phi_R = \phi_{R^d} \circ f\]
i.e. the following diagram commutes:
\[\begin{CD}
S @>\phi_{R^d}>>\cal S_{{R}^d} \\
@AfAA @AA{F_{R}}A \\
S @>\phi_{R}>> \cal S_{R}
\end{CD}\]
It has the following consequence: let $T_R$ be the point in the Teichmüller space $\cal T$ defined by $\phi_R$. Then:\footnote{Note that such a strong property as $F_R\circ\phi_R = \phi_{R^d} \circ f$ is not necessary to get $\sigma_f(T_{R^d}) = T_{R}$: a Thurston equivalence $\phi_{R^d}^{-1}\circ F_R\circ\phi_R \sim f$ would have been enough. However, it is as easy to directly get a $\phi_R$ satisfying the strong assumption.}
\[\sigma_f(T_{R^d}) = T_{R}.\]
In words, this says that the slow mating defines a path $R\mapsto T_R$ in the Teichmüller space, parameterized by $R\in]1,+\infty[$, and that Thurston's pullback map associated to the formal mating acts on it as the $d$-th root on the parameter.

\subsubsection{Squeezing the annulus}

In this section we do not need the polynomials $P_1$ and $P_2$ to be PCF, we just require them to have connected Julia sets. Given two values $R,R'$ it is quite natural to define the following non-holomorphic map:\footnote{Note the inversion: it will make equations described later easier to handle; it is regrettable but almost imposed on us by the fact that the composition notation $f\circ g$ reads backwards.}
\[ \Psi_{R',R}: \cal S_R \to \cal S_{R'}\]
defined as folows: map $z\in K_1$ in the chart $U_1(R)$ to the point of $\cal S_{R'}$ with the same coordinate $z$ in $U_1(R')$. Do the same for $K_2$. Map a point in the annulus $A(R)$ to the point on the same external ray but with potential multiplied by $\log R'/\log R$ (the potential may be measured in either the Böttcher coordinates of $P_1$ or that of $P_2$, or also with $\log |\Phi|$ that assigns $0$ to the equator, where $\Phi$ is the isomorphism from $A(R)$ to $1/R<|z|<R$; in all cases this gives the same result). This map is continuous and better: it is quasiconformal.

These maps are compatible with the dynamics:
\[F_{R'}\circ \Psi_{R',R} = \Psi_{{R'}^d,R^d} \circ F_R\]
i.e. the following diagram commutes:
\[\begin{CD}
\cal S_{R^d} @>\Psi_{{R'}^d,R^d}>>\cal S_{{R'}^d} \\
@A{F_R}AA @AA{F_{R'}}A \\
\cal S_{R} @>\Psi_{R',R}>> \cal S_{R'}
\end{CD}\]
and
\[\Psi_{R'',R} = \Psi_{R'',R'} \circ \Psi_{R',R}.\]
On the annuli $A(R)$,  $A(R')$, \ldots, these two identities basically follow from the fact that $F$ and $\Psi$ act as multiplications on the potential and on the external angle, and multiplications commute.

\subsubsection{A construction of the formal mating and of markings of $\cal S_R$.}\label{subsub:constr}

In fact we'll give two. In this section we do not need the polynomials $P_1$ and $P_2$ to be PCF, we just require them to have connected Julia sets.

The quick and dirty way consists in choosing a particular value of $R$, say $R=\exp(1)=e$ and using $S=\cal S_e$, $f=\Psi_{e,e^d} \circ F_e$, and letting the marking be: $\phi_R = \Psi_{R,e}$. The claims of Section~\ref{subsub:rel} follow at once.

For the second construction, map $\C$ to the unit disk by the non-conformal homeomorphism $\psi: z\mapsto z/(1+|z|)$. Let $S$ be the quotient of the disjoint union of two copies $\ov D_1$ and $\ov D_2$ of the closed unit disk, glued along their boundaries with $e^{i\theta}\in \partial D_1$ identified with $e^{-i\theta}\in \partial D_2$. We will identify $\ov D_1$ and $\ov D_2$ with their image in $S$. We first define a modification of $P_1$, call it $\wt P_1$ as follows: $\wt P_1=P_1$ on $K_1$ and for $z\in\C\setminus P_1$, $\wt P_1 (z)$ is the point of $\C\setminus K_1$ with the same potential as $z$, but with external angle multiplied by $d$. The map $\wt P_2$ is defined similarly. Then let $f$ be the conjugate of $\wt P_1$ by $\psi$ on $D_1$, the conjugate of $\wt P_2$ by $\psi$ on $D_2$ an be the multiplication by $d$ of the argument on the circle bounding $D_1$ and $D_2$.
Now for $z\in S$ let $\phi_R(z)$ be defined as follows. If $z\in D_1$ then let $w=\psi^{-1}(z)\in\C$; if $w\in K(P_1)$ then let $\phi_R(z)$ be the point of coordinates $w$ the chart $U_1$; if $w\notin K(P_1)$ then let $V$ be the potential of $w$ (it can be any real number in $]0,+\infty[$ since $w$ can be any complex number in $\C\setminus K(P_1)$) and let $\phi_R(z)$ be the point in the chart $U_1(R)$ whose coordinate is the point on the same external ray of $P_1$ as $w$, but with potential $\log(R)V/(1+V)$.

\begin{remark} In both cases we chose the markings so that
\[\phi_{R'} = \Psi_{R',R}\circ\phi_R.\]
It was not necessary. Note also that, as far as the Thurston characterization of rational functions is concerned, the particular dynamics of $f$ is not relevant, only its Thurston equivalence class is. Other constructions exist of $f$, for instance one where the map $f$ is $C^\infty$ on a $C^\infty$ sphere, and has an attracting equator. In this case, the marking cannot satisfy both $\phi_{R'} = \Psi_{R',R}\circ\phi_R$ and $\phi_{R^d}^{-1}\circ F_R\circ\phi_R = f$, but they still satisfy the following weaker form: $\phi_{R^d}^{-1}\circ F_R\circ\phi_R$ and $f$ are isotopic by an isotopy constant on $K_1$ and $K_2$. In particular they are Thurston equivalent if $P_1$ and $P_2$ are PCF.
\end{remark}

\subsection{Topological mating vs formal mating}\label{subsec:topovsform}

It may happen that the some ray equivalence classes contain several marked or critical points. Then the topological mating cannot be Thurston equivalent to the formal mating: the formal mating has to be corrected. Rees, Shishikura and Tan Lei have devised at least three ways of doing it, which probably have been proven equivalent (though the author could not find a written proof): see \cite{R}, \cite{S} and \cite{TS}. As far as the author understood, the corrected formal mating exists and is realizable (Thurston-equivalent to a rational map) if and only if the topological mating is conjugated to a rational map, and then the two rational maps are Thurston-equivalent (thus conjugated by a Möbius map if they are not flexible Lattès maps).

Even in degree $2$, there are pairs of post-critically finite polynomials which have an obstructed formal mating but still have a topological mating conjugated to a rational map: for instance the mating, studied in \cite{M2}, of $P=z^2+c$ with itself where $c$ is at the end of the Mandelbrot set external ray of argument $1/4$. The post-critical set of $P_1=P_2=P$ has three elements: $P(c),P^2(c),P^3(c)=P^4(c)$. The ray equivalence identifies $P_1^2(c)$ with $P_2^2(c)$ and $P_1^3(c)$ with $P_2^3(c)$, but no other point in the post-critical set of $P_1$ and $P_2$. Their topological mating is conjugate to the non-flexible Lattès map associated to multiplication by $1+i$ on the square lattice.

Note that the slow mating is a realization of the Thurston algorithm for the formal mating but not for the corrected formal mating. In the case of a mating that needs and has a correction, it is likely that under the slow mating, the marked points belonging to a same ray class will get closer and closer and tend to a single point\footnote{In the sense that they will be separated from the other marked points by an annulus of modulus tending to $+\infty$ separating them from the others, or equivalently by a geodesic of length tending to $0$, for the hyperbolic metric on the complement of the marked points: see Section~\ref{subsub:aboutpinch}}, and that the maps $F_R$ will converge to a map of the same degree (i.e.\ without loss of degree). We do not know if it has been proved completely.

\subsection{Levy cycles}

In his thesis \cite{L}, Silvio Levy proved the following theorem. Let $f$ be a post-critically finite topological ramified self-cover of the sphere:
\begin{theorem}[Levy]
If $f$ has degree $2$ then it has a Thurston obstruction if and only if it has a Levy cycle.
\end{theorem}
A Levy cycle is a numbered multicurve $\gamma_0$, $\gamma_1$, \ldots, $\gamma_n=\gamma_0$ such that for all $k$, $\gamma_{k+1}$ is isotopic rel.\ $P_f$ to a component of $f^{-1}(\gamma_k)$ on which $f$ has degree $1$.
The curves of a Levy cycle form an obstruction so one implication of the theorem is trivial. 
Note that the Levy character of a given numbered multicurve cannot be read off from its Thurston matrix.

\subsubsection{Polynomial matings and Levy cycles}

Formal matings which need and have a correction have a special class of Levy cycles called removable Levy cycles. 

Concerning matings of post-critically finite polynomials, the situation is completely understood in degree $2$:
using Levy's theorem, Tan Lei proved:
\begin{theorem}[Tan Lei]
The pairs of post-critically finite polynomials of degree $2$ whose topological mating is not equivalent to a rational map are those who belong to conjugate limbs of the Mandelbrot set. Moreover, their ray equivalence classes contain loops and the quotient topological space in the definition of the topological mating is not a sphere.
\end{theorem}

However, not only Levy's theorem has counter examples in degree $\geq 3$ but it is has counter examples among formal matings:
\begin{theorem}[Shishikura, Tan Lei, \cite{TS}]
There exist a pair of post-critically finite polynomials of degree $3$ whose formal mating is obstructed but has no Levy cycles. Moreover, the corresponding quotient topological space is homeomorphic to a sphere.
\end{theorem}
Having no Levy cycles, this mating is not correctable, so the topological mating of $P_1$ and $P_2$ is not conjugate to a rational map. The author believes one can prove it implies there is no way to normalize the spheres $\cal S_R$ so that the rational maps $F_R$ converge as $R\tend 1$, even for a subsequence $R_n\tend 1$. In this particular example, this is supported by the experimentations presented in Section~\ref{sec:conflim}. It is remarkable, though, that the quotient topological space is still a sphere: this is basically because the ray equivalence relation is closed and the classes contain no loop (see \cite{TS}).

\subsubsection{Description of the example of Tan Lei and Shishikura}

Their example is explicit.
Post-critically finite polynomials can be characterized, up to complex-affine conjugacy, by the augmented Hubbard trees. Recall that the augmented Hubbard tree is the Hubbard tree together with some angle information at some vertices. Remember also that to mate polynomials, we must align the external angles. This can be done by choosing which fixed external ray is the one of angle $0$ (this amounts to choosing monic centered polynomials). In terms of the combinatorial information, this means giving some enough supplementary information to distinguish it.\footnote{This is not very important, but note that for polynomials commuting with a rotation, one only needs to characterize a class of fixed external ray modulo this rotation.}

It turns out that less information is often sufficient. Here the polynomials $P_1$ and $P_2$ in the example of Tan Lei and Shishikura can be charaterized up to $\C$-affine conjugacy by the data on Figure~\ref{fig:He}. The monic centered polynomials are uniquely determined by the data on Figure~\ref{fig:Hee}.
In particular $P_1$ has critical points $x$, $y$ and a $3$-cycle $x\mapsto y\mapsto z \mapsto x$, and $P_2$ has a double critical point (thus of local degree $3$) of period $3$.

\begin{figure}%
\begin{tikzpicture}
\draw (-4,-1.2) node {$P_1$};
\draw (4,-1.2) node {$P_2$};
\node at (0,0) {\includegraphics{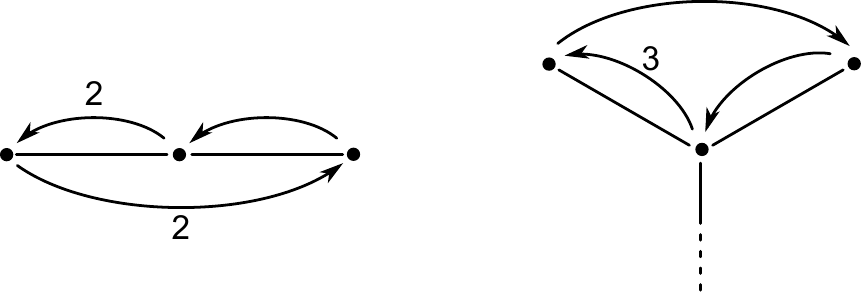}};
\end{tikzpicture}
\caption{Hubbard trees $H$ of $P_1$ and $P_2$. The number over the arrow indicates the local degree at the starting point when it is not one. The dotted line indicates how some particular edge of the first pre-image of $H$ branches on $H$.}
\label{fig:He}
\end{figure}

\begin{figure}%
\begin{tikzpicture}
\node at (0,0) {\includegraphics{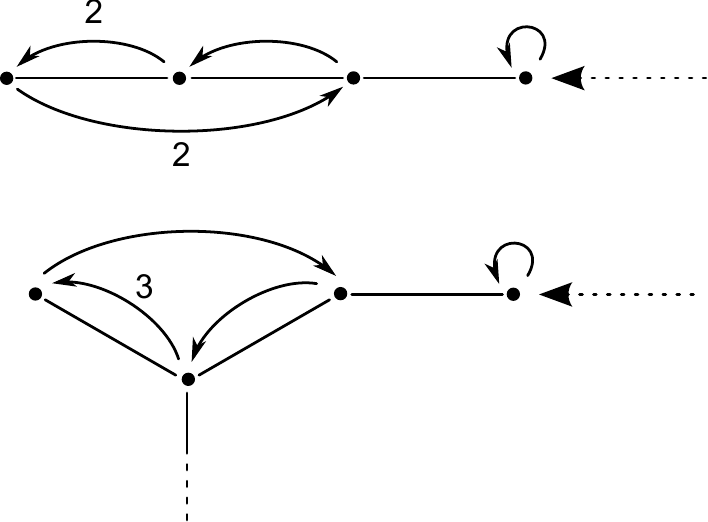}};
\end{tikzpicture}
\caption{Here we added the endpoint of the external ray $\cal R$ of angle $0$ of the each monic polynomial, together with $\cal R$ pictured in dotted lines. This data is enough, at least on these examples, to uniquely determine the monic centered polynomials $P_1$ and $P_2$. We did not indicate the other fixed external ray, which has angle $1/2$.}
\label{fig:Hee}
\end{figure}

The polynomials are
\bEA P_1 &=& z^3+az+b
\\ P_2 &=& z^3+c
\eEA
with
\bEA 
& &  a=-3x^2,\quad b=2x^3-x, \quad 32x^8-24x^6+2x^2-1=0,\quad x\approx 0.8445,
\\ & & c^8+3c^6+3c^4+c^2+1=0,\quad c\approx -0.264+1.260i.
\eEA

And their Julia sets are illustrated on Figure~\ref{fig:JP1P2}.

\begin{figure}%
\begin{tikzpicture}
\node at (0,0) {\includegraphics[width=12cm]{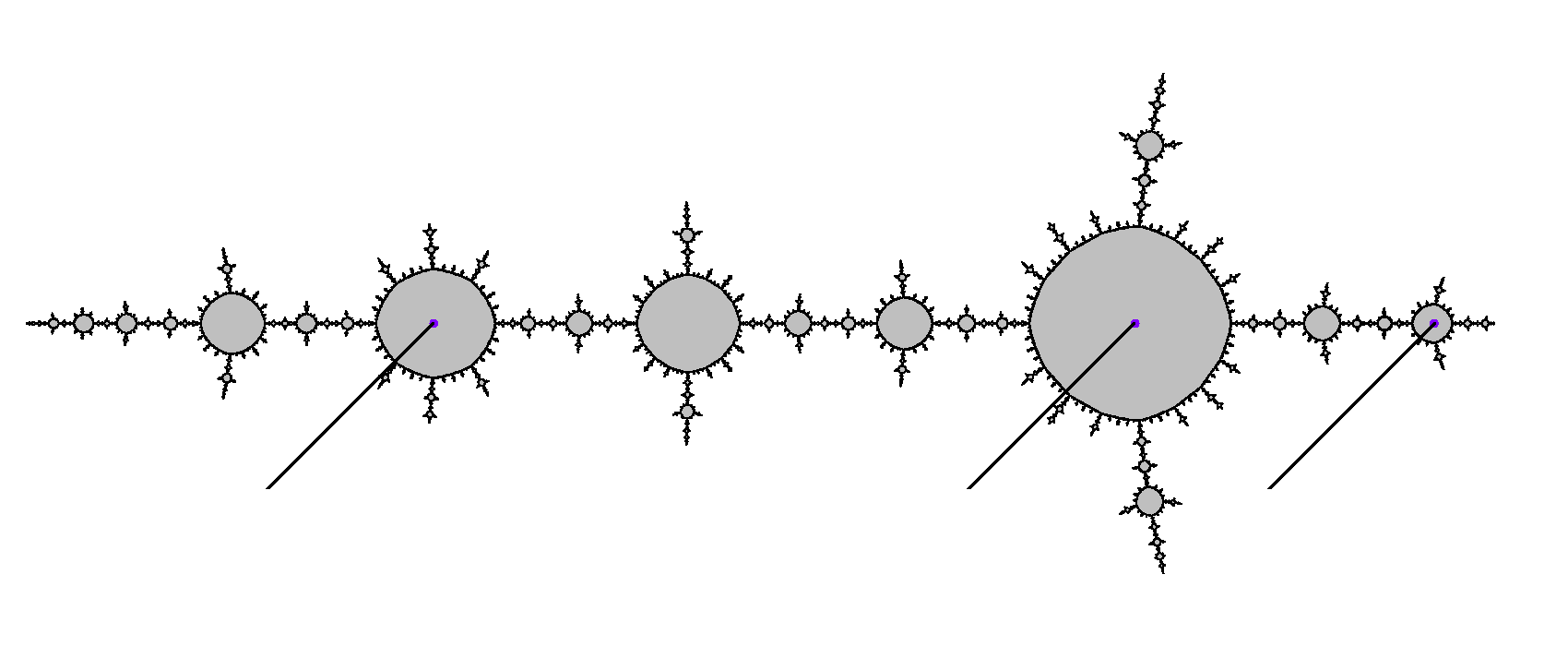}};
\node at (-4.1,-1.5) {$y$};
\node at (1.3,-1.5) {$x$};
\node at (3.6,-1.5) {$z$};
\node at (0,-8) {\includegraphics[width=12cm]{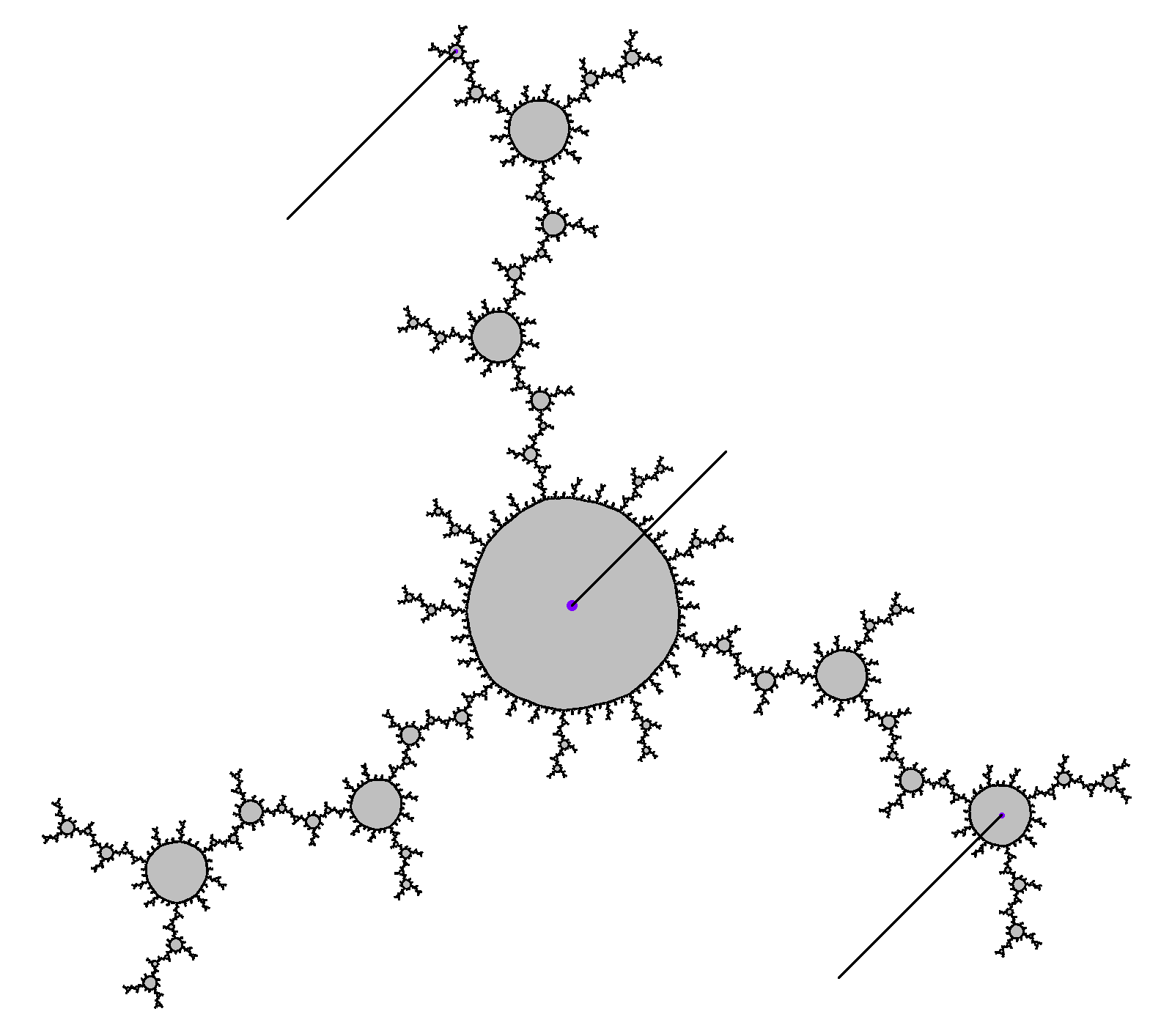}};
\node at (-2.8,-5.2) {$c=P_2(0)$};
\node at (1.9,-7.1) {$0 = P_2^3(0)$};
\node at (2.5,-13) {$P_2^2(0)$};
\end{tikzpicture}%
\caption{The Julia sets of $P_1$ and $P_2$.}
\label{fig:JP1P2}
\end{figure}
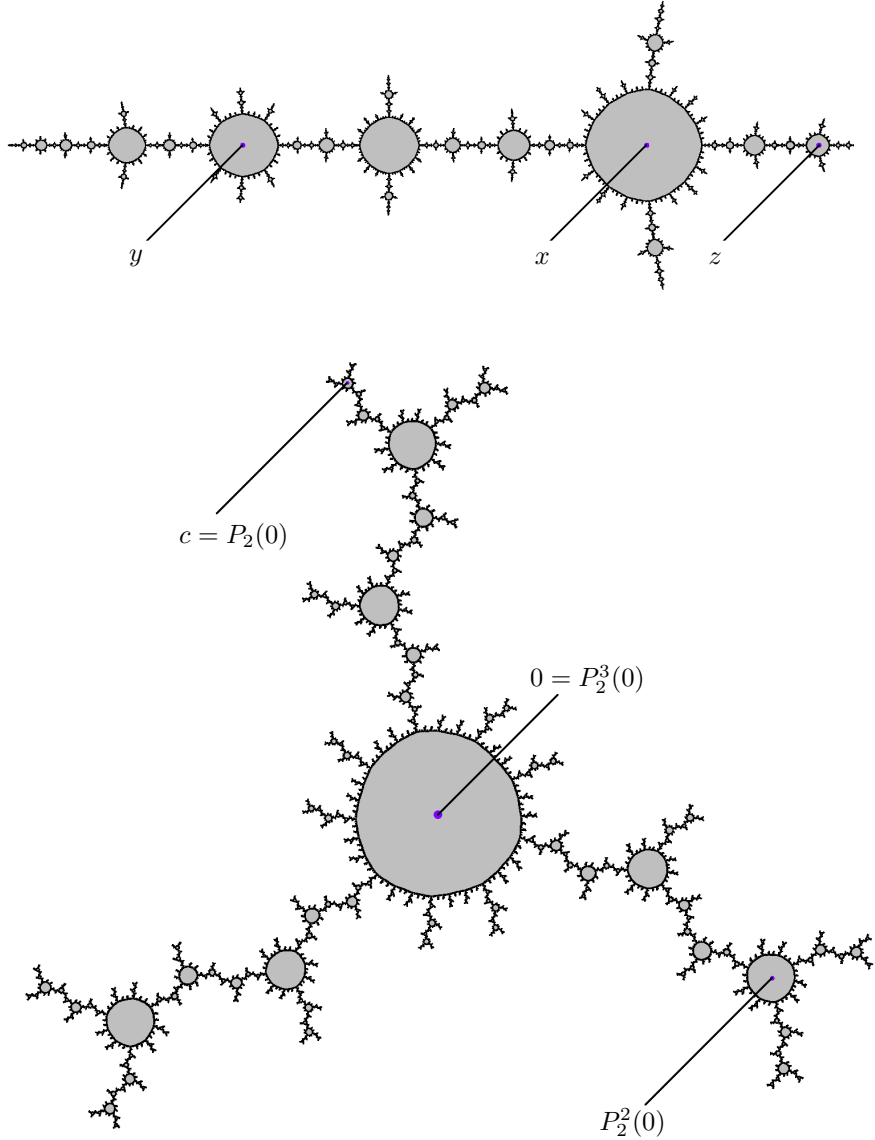

\subsubsection{An obstruction for the example}

\begin{figure}%
\begin{tikzpicture}
\node at (-3.2,1.9) {$a$};
\node at (1.4,2.0) {$b$};
\node at (0,0) {\includegraphics{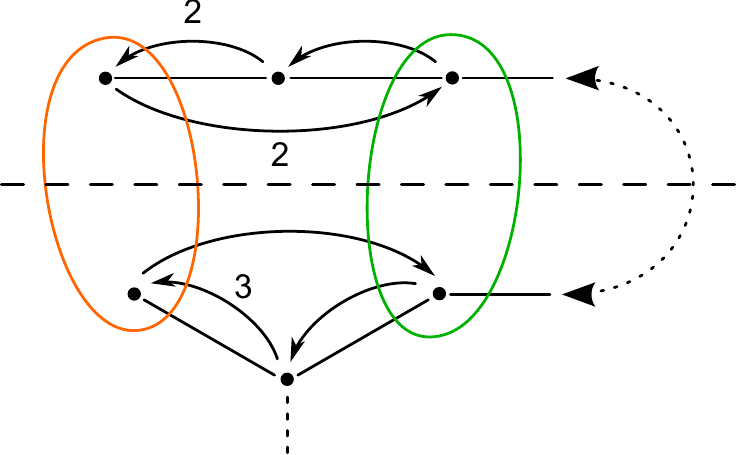}};
\end{tikzpicture}
\caption{The multicurve $\{a,b\}$ is an obstruction for the formal mating of $P_1$ and $P_2$. The dashed line represents the equator and the dotted line the external angle of argument $0$.}
\label{fig:obs}
\end{figure}

\begin{figure}%
\begin{tikzpicture}
\node at (0,0) {\includegraphics[width=10cm]{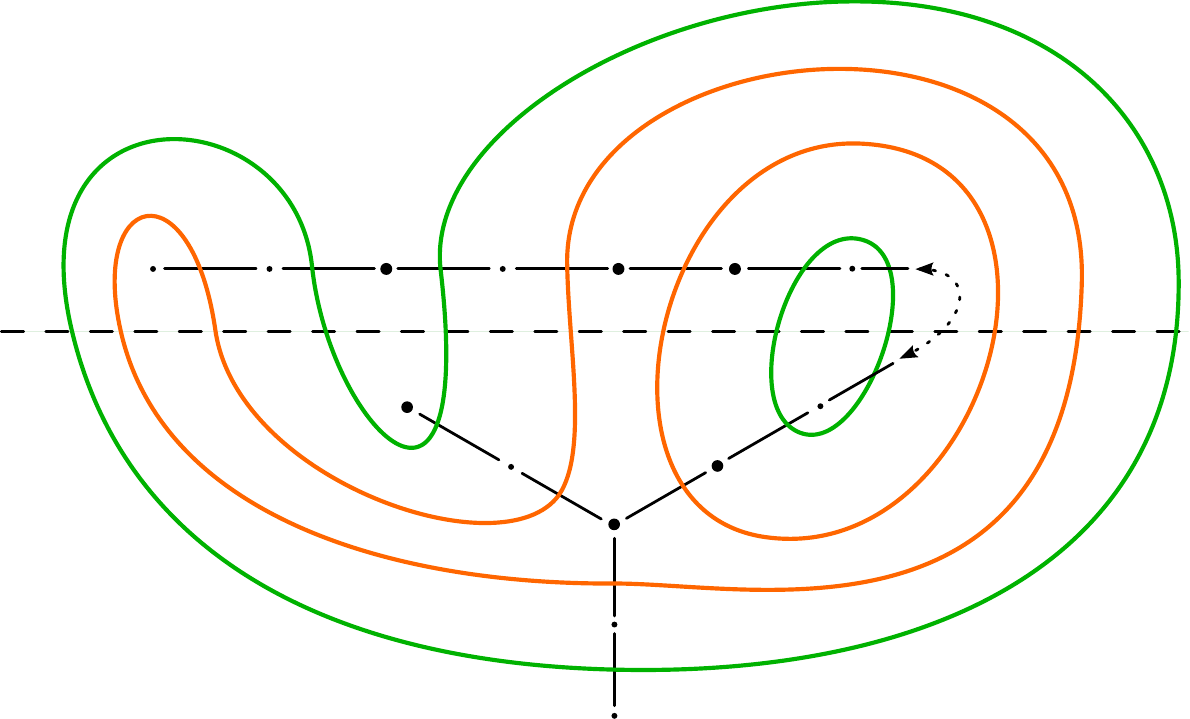}};
\end{tikzpicture}
\caption{Preimage by the formal mating of the multicurve $\{a,b\}$ and of the Hubbard trees. The external ray of angle $0$ is also indicated. Big dots represent the marked points, i.e. the post-critical set of the formal mating $f$. Here, the preimage of the big dots by $f$ is the union of the big and the small dots.}
\label{fig:preimg}
\end{figure}

The formal mating $f$ has a Thurston obstruction, a multicurve $\Gamma=\{a,b\}$, illustrated on Figure~\ref{fig:obs}. On Figure~\ref{fig:preimg} we put the preimage of $\Gamma$. From this it is easily seen that the Thurston matrix associated to $\Gamma$ in the basis $(a,b)$ is:
\[\left[\begin{matrix} 1/2 & 1/2 \\ 1 & 0
\end{matrix}\right].
\]
This matrix has spectrum $\{1/2,1\}$ therefore $\Gamma$ is indeed an obstruction.

\subsubsection{About pinching curves}\label{subsub:aboutpinch}

This section may be skipped by the reader familiar with the theory. The interested reader may look at \cite{H} or \cite{BFLSV}.

\begin{figure}%
\begin{tikzpicture}
\node at (0,0) {\includegraphics[width=8cm]{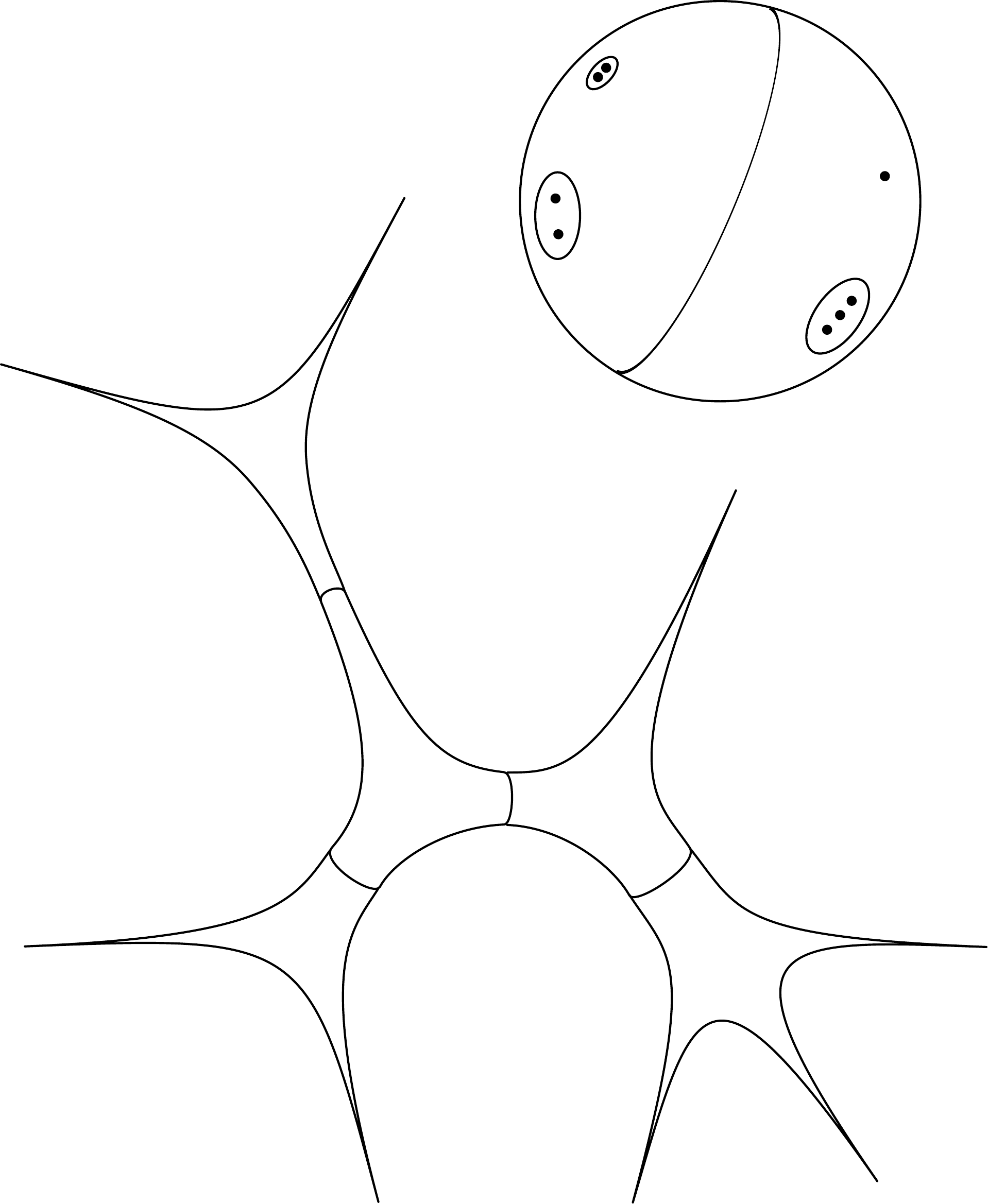}};
\end{tikzpicture}
\caption{A configuration of points on the Riemann sphere $\wh\C$, together with a multicurve, supposedly homotopic to the short geodesics. Below it, a representation of what the hyperbolic metric should look like; the 8 cusps correspond to the 8 marked points, and they are infinitely long.}
\label{fig:tree}
\end{figure}

\begin{figure}%
\begin{tikzpicture}
\node at (0,0) {\includegraphics[width=8cm]{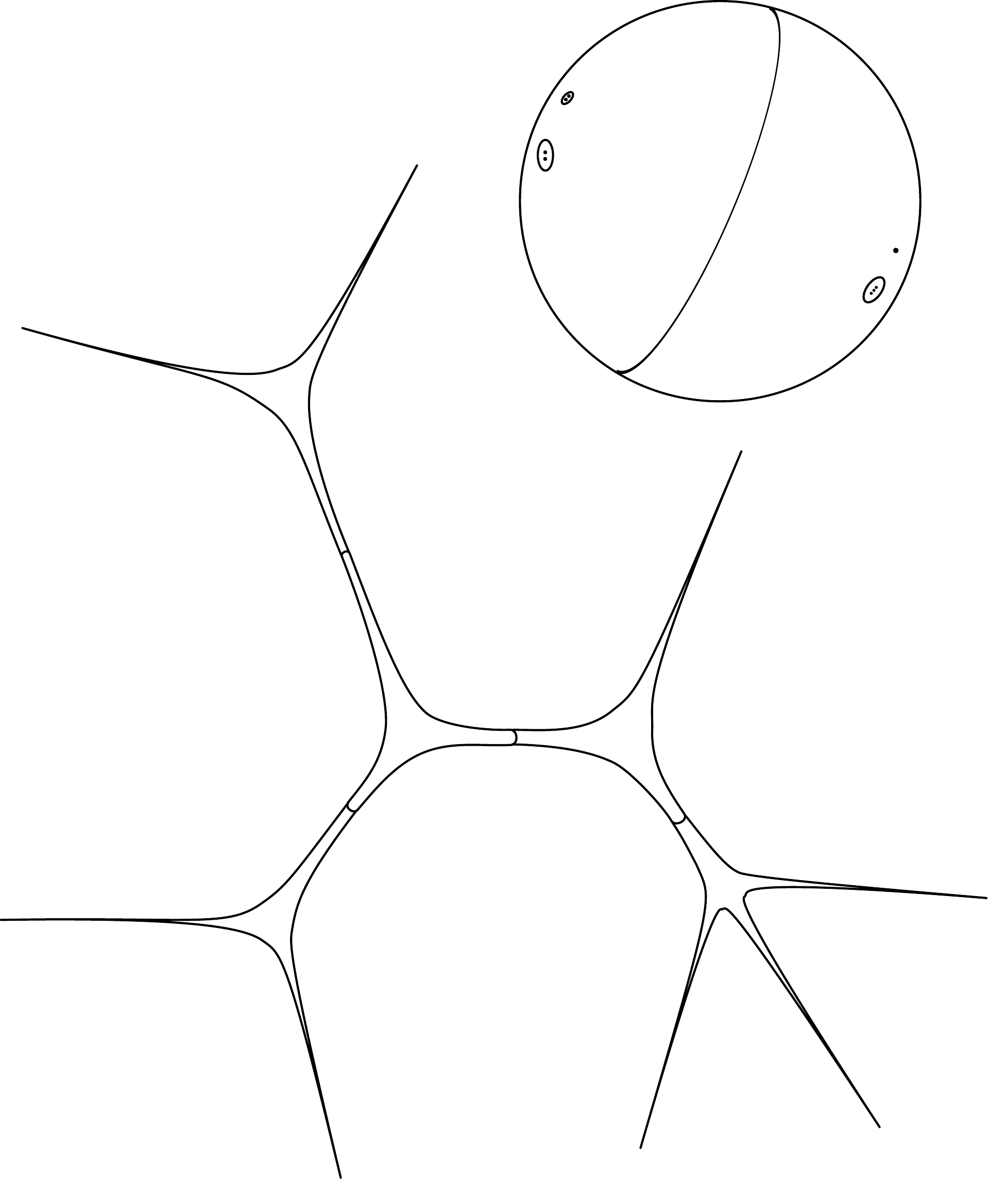}};
\end{tikzpicture}
\caption{Same object but with shorter geodesics. The hyperbolic version is represented with a bigger scale than on the previous figure. This picture and the previous one are not meant to be precise, but to give a sense of the hyperbolic metric.}
\label{fig:tree2}
\end{figure}

We denoted earlier $f: S\to S$ the formal mating and $P_f$ its postcritical set. 
Recall: a point $[\phi]$ in the Teichmüller space $\cal T=\cal T(S,P_f)$ is an equivalence class of map $\phi \to \wh\C$ for the appropriate equivalence relation.\footnote{I like to think of the Teichmüller space $\cal T(S,P_f)$ as the set of isomorphism classes of Riemann surfaces marked by $S$, the marking being flexible between the points of $P_f$.
} The elements of $\phi(P_f)$ are called the marked points. Thurston's pull-back map associates to $[\phi]$ a new point $\sigma_f([\phi])$.

The \emph{moduli space} $\cal M(P_f)$, more tractable in many respects, is the set of maps from $P_f$ to $\wh\C$ modulo post-composition by a Möbius map. We will call its elements \emph{configurations}. There is a projection $\pi:\cal T(S,P_f)\mapsto \cal M(P_f)$ that maps the class of $\phi$ to the class of its restriction to $P_f$. The theory of Thurston's pull-back map has the following proposition: if a Thurston map $f$ has hyperbolic orbifold, then $\sigma_f^n([\phi])$ diverges in $\cal T(S,P_f)$ if and only if $\pi(\sigma_f^n([\phi]))$ leaves every compact of $\cal M(P_f)$. 
This means that whichever normalization one chooses for the configuration of points on $\wh\C$ given by $\pi(\sigma_f^n([\phi]))$, passing to a subsequence  there will be at least two points with the same limit.

The fact that one has normalizations to choose makes the discussion difficult and it is much more pleasant to work with the more rigid structure provided by the hyperbolic metric, a.k.a.\ Poincaré metric, of the complement of the marked points, i.e.\ of $U=\wh\C\setminus\phi(P_f)$. We will explain why
Let us first state here a few facts of the geometry of the Poincaré metrics (see \cite{H}): in each of the two components of the complement in $\wh \C$ of a simple closed geodesic in $U$, there is at least two marked points. Conversely, for each simple closed curve $\gamma$ in $U$, such that both component of $\wh \C\setminus\gamma$ contains at least two marked points (such curves are called \emph{non peripheral}), there is a unique geodesic homotopic to $\gamma$ (and it is the shortest curve homotopic to $\gamma$). 
Endow temporarily $\wh\C$ with a spherical metric, using for instance stereographic projection from the Euclidean $2$-sphere. There exists a map $h(\epsilon)\underset{\epsilon\to 0}\tend 0$ such that for a configuration $C\in\cal M(P_f)$ with a simple closed geodesic of hyperbolic length $\leq\epsilon$, whichever representative (normalization) $\phi \in C$ is chosen, there is at least one side of the geodesic where the marked points are grouped in a bunch of size at most $h(\epsilon)$, and there exists $\phi\in C$ such that each group lives at spherical distance $\leq h(\epsilon)$ of respectively $0$ and $\infty$.
Conversely there is a map $g(\epsilon)\underset{\epsilon\to 0}\tend 0$ such that if $C\in\cal M$ has a representative for which the marked points are split in two groups, each with at least two points, and one at distance $\leq g(\epsilon)$ from $0$ and the other at distance $\leq g(\epsilon)$ from $\infty$, then there is a simple closed geodesic separating the groups and of hyperbolic lenght $<\epsilon$ in $U$.

As a corollary: a sequence of configurations $[\phi_n]\in \cal M(P_f)$ will leave every compact in $\cal M(P_f)$ if and only if there exists a sequence of simple closed $\gamma_n$ geodesics in $U_n=\wh\C\setminus\phi_n(P_f)$ whose lengths tends to $0$.

Let us state a few more geometric facts. Two disjoint non peripheral curves are homotopic to geodesics which are either disjoint or equal. In particular a multicurve in $U=\wh\C\setminus\phi(P_f)$ has a privileged representative, which is the collection of simple closed geodesics homotopic to its curves. There is a universal constant $L_0>0$ such that no two different geodesics of length $<L_0$ can intersect (also, but we will not use this fact, a closed geodesic with length $<L_0$ is necessarily simple). If $P_f$ has $k$ points then there is at most $k-2$ different simple closed geodesics of length $<L_0$.

Pieces: Hence endowing $\wh\C\setminus\phi(P_f)$ with its hyperbolic metric, gives us a way to separate the marked points into groups, by cutting the sphere along those few simple closed geodesics that are small (with a notion of small chosen according to the use, for instance shorter than $L_0$). In fact we get more than just groups of point: all pieces do not necessary contain a marked point; the graph whose vertices is pieces and edges are curves separating two pieces forms a tree. See Figures~\ref{fig:tree} and~\ref{fig:tree2}.

Tree: More generally cut a sphere with marked points along a multicurve. The sphere is split it into open pieces $U_i$. Let $U'_i$ be the pieces minus the marked points. Each $U'_i$ has at least three boundary components, which may be points or curves. Sometimes, there is no points in some pieces $U_i$. Nevertheless, it is possible to single out any $U_i$ by choosing three marked points. More precisely, for any three marked points, there is a unique $U'_i$ which, when removed from the sphere, disconnects the three points (this is somewhat related to the tree structure associated to the pieces). For every component, there is a (not necessarily unique) set of three points that it disconnects. 

\begin{figure}%
\begin{tikzpicture}
\node at (0,0) {\includegraphics[width=6cm]{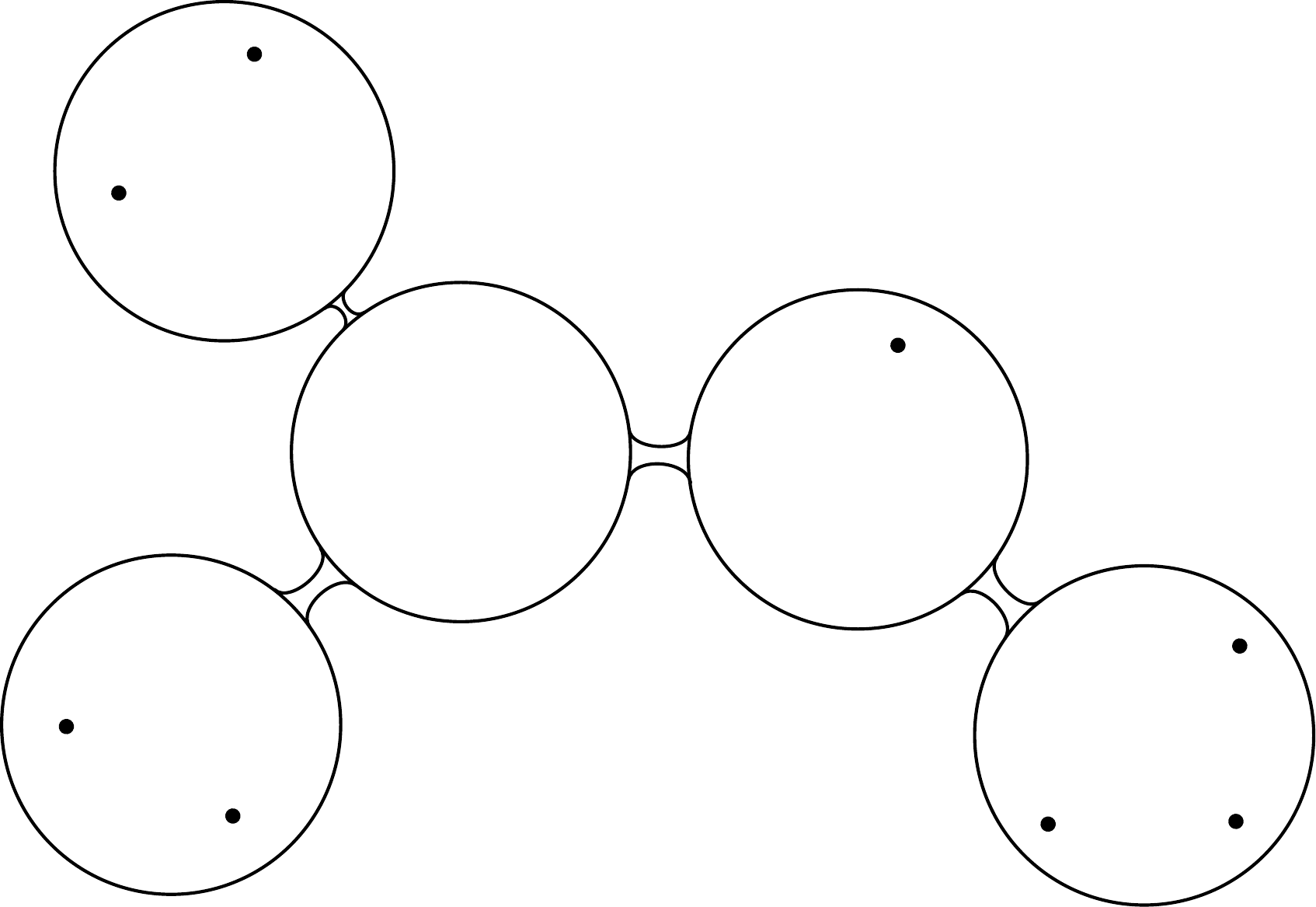}};
\end{tikzpicture}
\caption{Another (schematic) conformal model of the punctured sphere with short geodesics: a tree of sphere. A Möbius map will zoom on each piece, represented here as a sphere. Small disks can be removed from each sphere and tubes can be added (there is flexibility in the choice of their shape).}
\label{fig:spheretree}
\end{figure}

\begin{figure}%
\begin{tikzpicture}
\node at (0,0) {\includegraphics[width=6cm]{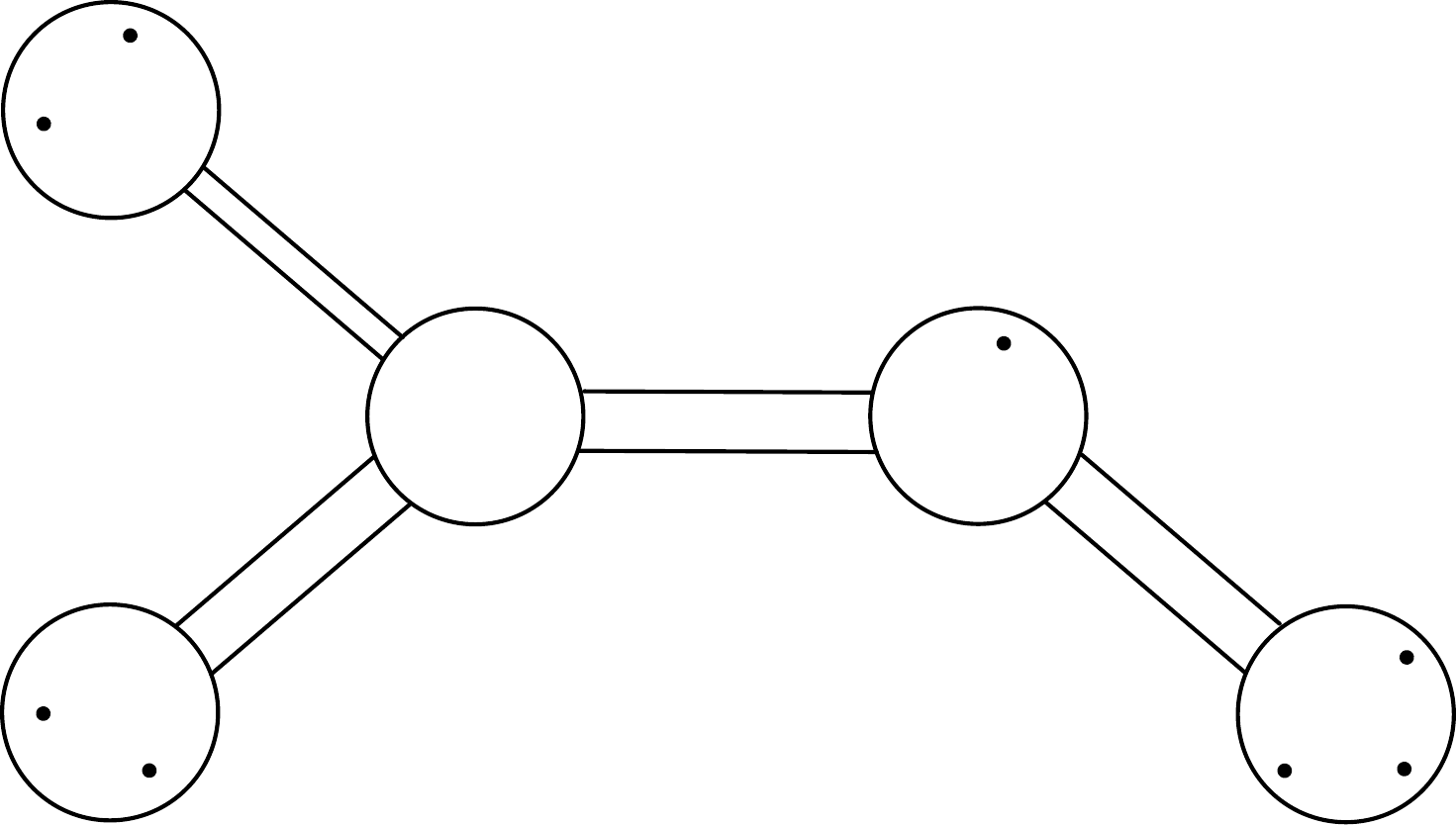}};
\end{tikzpicture}
\caption{Yet another conformal model.}
\label{fig:spheretree2}
\end{figure}

Zoom: Let us get back to $\wh\C\setminus\phi(P_f)$ with its conformal structure and its hyperbolic geometry. Cut it along small simple closed geodesics (for a notion of small that depends on the use). Select three marked points, with some order, and compose $\phi$ with the unique Möbius map sending them respectively to $0$, $1$, $\infty$. This normalization gives a new representative of the configuration which focuses (or kind of zooms) on the piece associated to the selected triple. The smaller the geodesics bounding the piece, the more concentrated the marked point on the other side of the geodesic, viewed on the normalized Riemann sphere. 

Figure~\ref{fig:spheretree} sums this up by a conformal presentation of the Riemann surface taking pinching curves into account. These trees of spheres are the usual point of view used for the definition of a standard compactification of the moduli space, see~\cite{BFLSV}.
Figure~\ref{fig:spheretree2} gives yet another way to conformally represent things. It puts the emphasis both on the sphere and on the funnels between them. The latter are annuli with a big modulus, so they can also be viewed as tunnels. We insist on them because we will present some images that focus on a tunnel instead of focusing on the big spheres.

\subsubsection{The canonical obstruction}

Let us say that two simple closed curves $\gamma$ and $\gamma'$ in $\wh\C\setminus P_f$ are homotopically transverse if there is no curve homotopic to $\gamma$ and disjoint from $\gamma'$.
Kevin Pilgrim in \cite{P} precised how the curves pinch under iteration of Thurston's pull-back map:

\begin{theorem}[Pilgrim]\label{thm:pilgrim} If $f$ is Thurston map with hyperbolic orbifold and is not realizable (so it has obstructions), then there exists a multicurve $\Gamma$, called the \emph{canonical obstruction}, such that $\forall[\phi]\in\cal T$, let $[\phi_n]=\sigma^n([\phi])$ and $U_n =\wh\C\setminus \phi_n(P_f)$: 
\begin{itemize}
\item $\Gamma$ is not empty,
\item for all curve $\gamma\in\Gamma$, the (simple closed) geodesic in $U_n$ homotopic to $\phi_n(\gamma)$ has a length tending to $0$ as $n\tend +\infty$,
\item $\exists N\in\N$ and $L>0$ such that $\forall n\geq N$, these geodesics are the only simple closed geodesics shorter than $L$,
\item $\Gamma$ is an obstruction,
\item $\Gamma$ is $f$-stable and surjective\footnote{Surjective (non official terminology): It means that all curve in $\gamma$ is homotopic to at least one curve in $f^{-1}(\Gamma)$. To a multicurve is associated a directed graph with vertices the curves and arrows $a \to b$ whenever $b$ is homotopic to a component of $f^{-1}(a)$. The multicurve is called surjective whenever all $b$ have an incoming arrow, in analogy to functions. It does not imply that the matrix is surjective.},
\item no curve in any obstruction can be homotopically transverse to a curve of $\Gamma$,
\end{itemize}
The canonical obstruction is unique up to homotopy.
\end{theorem}

It shall be emphasized that there exists examples with hyperbolic orbifold having obstructions that do not pinch. Still, these examples have canonical obstructions that do pinch.

One last remark: even though the hyperbolic geometry is a pleasant way to formulate things, the core of the proofs of Thurston's and Pilgrim's theorems use moduli of annuli, and this language could have been used instead.

\subsubsection{Does the obstruction pinch?}\label{subsub:pinch}

We prove in this section that in the example of Shishikura and Tan Lei, the given obstruction $\Gamma=\{a,b\}$ is the canonical obstruction.

In all this section, homotopic means homotopic in the complement of $P_f$ and isotopic means isotopic rel.\ $P_f$.

\begin{figure}%
\begin{tikzpicture}
\node at (0,0) {\includegraphics[width=7cm]{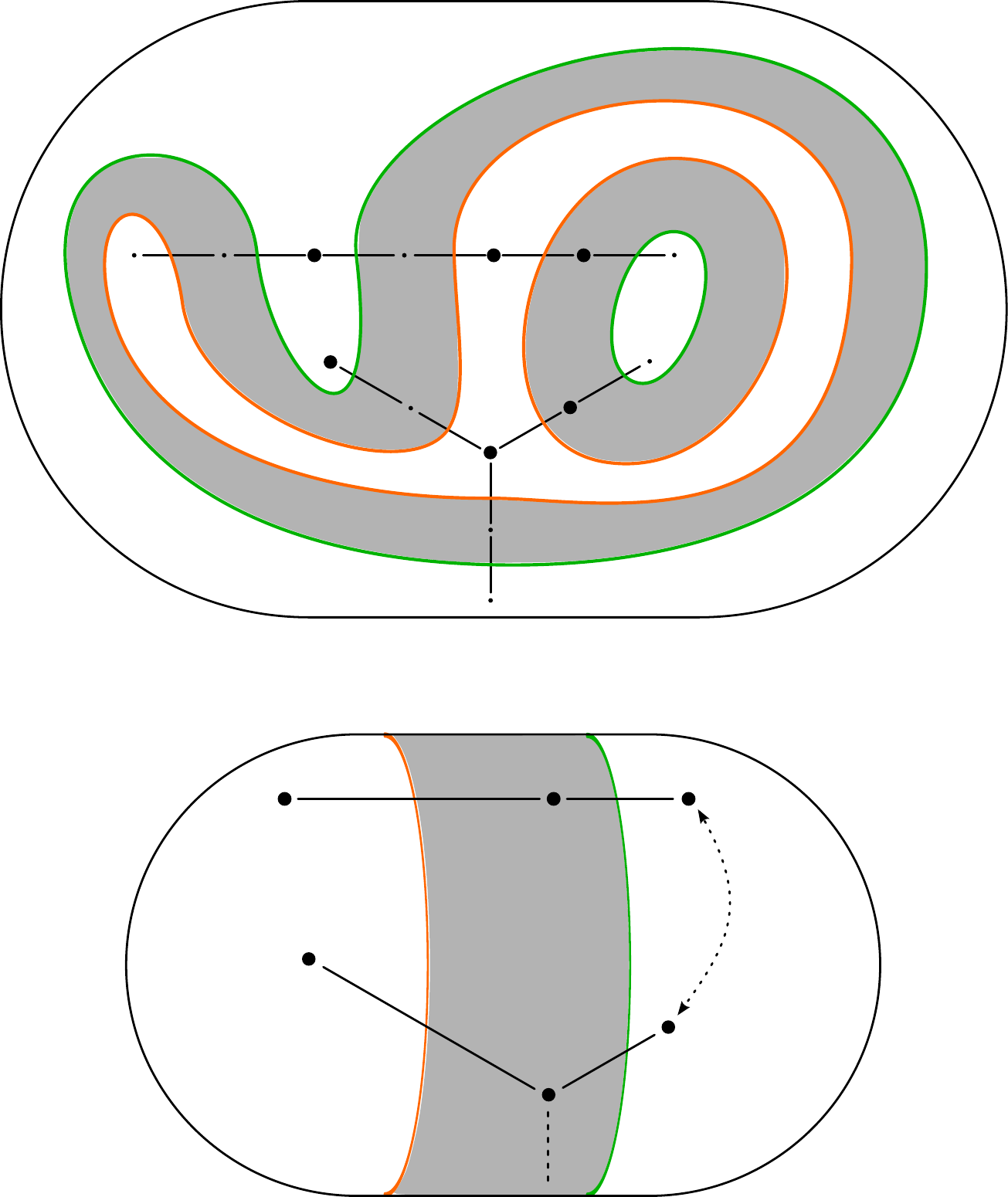}};
\draw (-3,-1) node[anchor=south east, inner sep=0pt] {$A_1$} -- (-2,-2);
\draw (-1.25,-.75) node[anchor=south east, inner sep=0pt] {$A_2$} -- (-0.25,-1.75);
\draw (3,-1) node[anchor=south west, inner sep=0pt] {$A_3$} -- (2,-2);
\draw[->] (0,-.2) -- node[right] {$f$} (0,-0.8);
\end{tikzpicture}
\caption{Illustration for Lemma~\ref{lem:ab}.}
\label{fig:cut}
\end{figure}
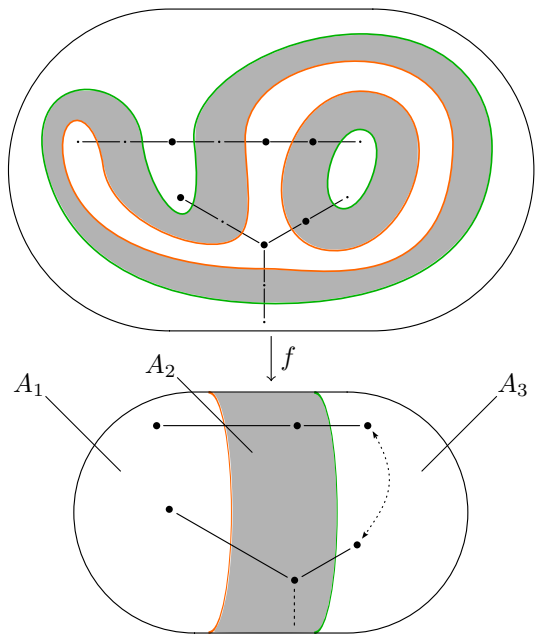

\begin{lemma}\label{lem:ab} Let $\gamma$ be a non-peripheral simple closed curve disjoint from $a$ and $b$. Then all components $\gamma'$ of $f^{-1}(\gamma)$ are either peripheral or homotopic to $a$ or $b$.
\end{lemma}
\begin{proof}
The multicurve $\Gamma=\{a,b\}$ cuts the sphere into three parts $A_1$, $A_2$, $A_3$, each of which contains exactly two marked points. The curve $a$ separates $A_1$ from $A_2$ and $b$ separates $A_2$ from $A_3$. See Figure~\ref{fig:cut}.
The curve $\gamma$ is contained in one of these components.
\begin{enumerate}
\item If $\gamma\subset A_1$, since $A_1\setminus P_f$ only has three boundary components, which are two marked points and the curve $a$, all non homotopically trivial simple closed curve in $A_1\setminus P_f$ is either homotopic to a small loop around a marked point, and thus peripheral, or homotopic to $a$. In particular $\gamma$ is homotopic to $a$. The preimage of $a$ has two components, one homotopic to $a$ and the other to $b$.
\item Similarly, if $\gamma\subset A_3$, then it is homotopic to $b$, which has two preimages, one homotopicto $a$ and the other which is homotopically trivial.
\item Last, if $\gamma\subset A_2$, note that $A_2$ has two preimages by $f$, one which is an annulus containing no marked points and whose equator is homotopic to $a$, and one which is an annulus containing two marked point and isotopic to a subset of $A_3$. If $\gamma'$ is contained in the first one, it is either null-homotopic or homotopic to $a$. If $\gamma$ is contained in the second one, then it is homotopic via the isotopy to a non peripheral curve contained in $A_3$, and by the same arguments as above, it is thus homotopic to $b$.
\end{enumerate}
\end{proof}

Let $\cal C$ be the union of the homotopy classes in the canonical obstruction. Denote $\gamma\to\gamma'$ whenever there $\gamma'$ is homotopic to a component of $f^{-1}(\gamma)$. Surjectivity of the canonical obstruction tells that all curves $\gamma\in \cal C$ are homotopic to a component of the preimage of a curve in $\cal C$. By the last point of Theorem~\ref{thm:pilgrim}, the latter has a representative disjoint from $a$ and $b$. Since $\gamma$ is not peripheral, the lemma tells us $\gamma\sim a$ or $\gamma\sim b$. So the canonical obstruction is homotopically contained in $\Gamma$. Moreover since the canonical obstruction is not empty, $a$ or $b$ must be in $\cal C$. Last, by $f$-stability, and since $a\to b$ and $b\to a$, we know that $a$ and $b$ belong to $\cal C$. So $\Gamma$ is the canonical obstruction.

\section{Conformal limit of the example of Shishikura and Tan Lei.}\label{sec:conflim}

In this section we show conformally accurate pictures of the slow mating on the Riemann sphere. We then give a conjectural explanation  of what is going on, and finish with a few more pictures.

\subsection{Conformally correct computer generated pictures.}

We need to give practical labels to the post-critical points of the formal mating $f$ and to the marked points. Let $x$ and $y$ denote the critical point of $f$ corresponding to the critical points $x$ and $y$ of $P_1$. Let $y_1$ be the point corresponding to $P_1(y)$. Recall that $x\mapsto y\mapsto y_1 \mapsto x$. Let $c_0$ denote the critical point of $f$ corresponding to the critical point $0$ of $P_2$. Let $c_1$ and $c_2$ correspond to $c=P_2(0)$ and $P_2^2(0)$. Recall that the marked points on $\cal S_R$ are the images of the post-critical points of $f$ by the marking $\phi_R:S\to\cal S_R$. We will use the same labels for the marked points as for the post-critical points of $f$.

\begin{figure}%
\scalebox{.95}{%
\begin{tikzpicture}
\node at (-3.5,0) {\includegraphics[width=6cm]{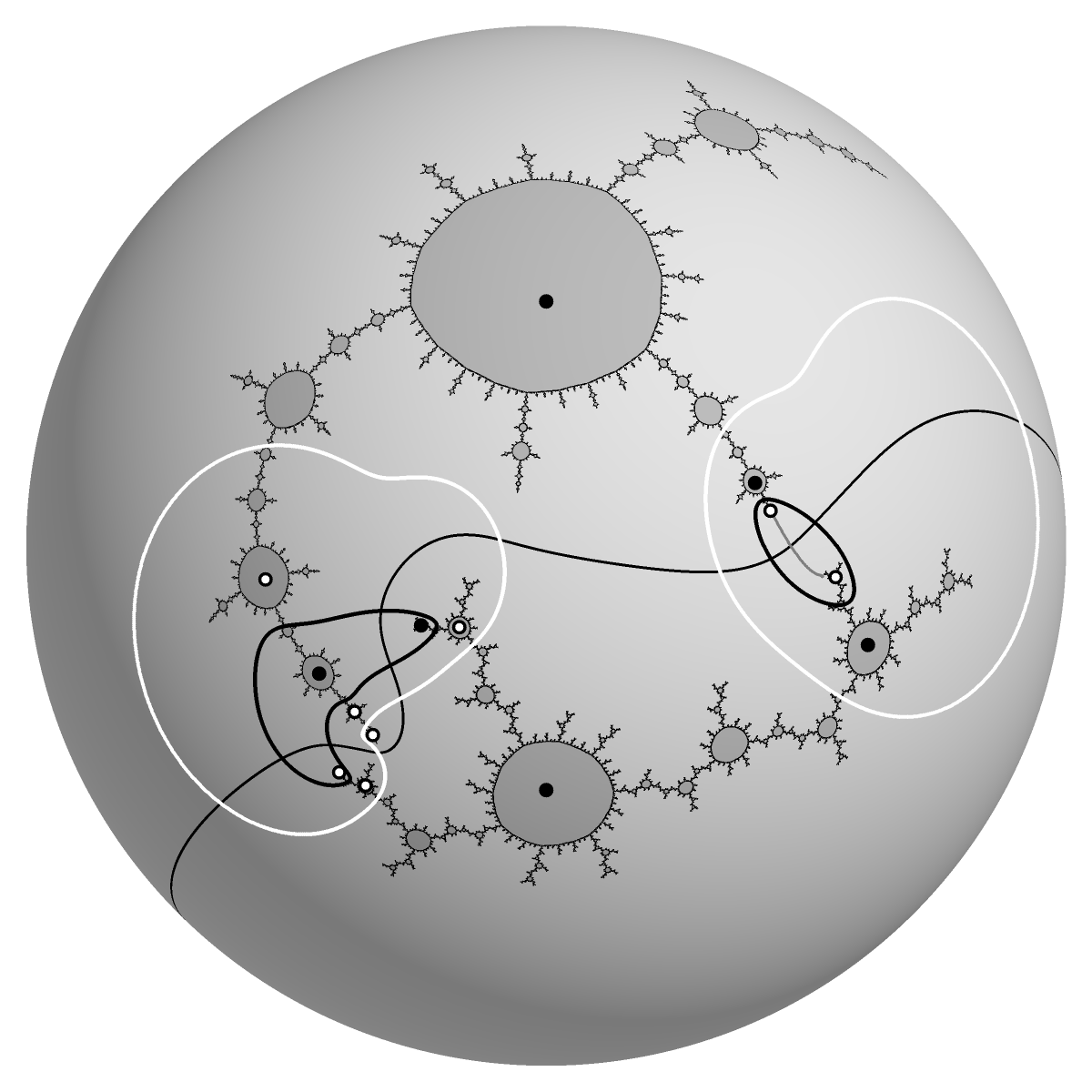}};
\node at (3.5,0) {\includegraphics[width=6cm]{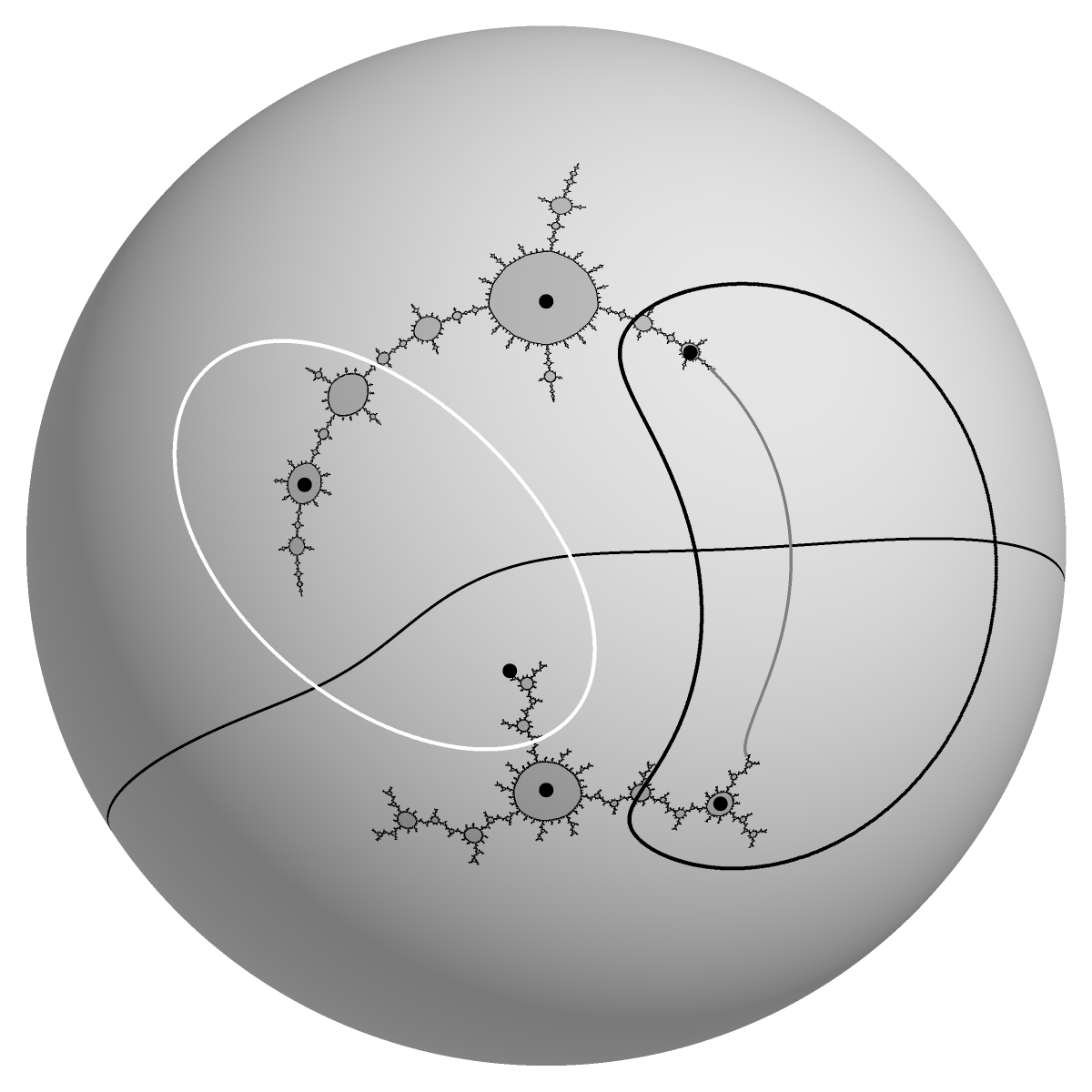}};
\draw[->] (-0.5,0) -- node[above]{$F_R$} (0.5,0);
\node at (-2.4,-3.1) {$\cal S_R$};
\node at (4.6,-3.1) {$\cal S_{R^3}$};
\draw (1.9,3) node[left] {$x$} -- (3.4,1.5);
\draw (0.5,2.0) node[left] {$y$} -- (2.0,0.5);
\draw (5.9,2.7) node[right] {$y_1$} -- (4.4,1.2);
\draw (2,-2.85) node[left] {$c_0$} -- (3.4,-1.45);
\draw (1.45,-2.5) node[left] {$c_1$} -- (3.15,-0.8);
\draw (2.9,-3.1) node[below] {$c_2$} -- (4.4,-1.6);
\end{tikzpicture}%
}
\caption{Conventions: In both pictures, the gray line is the external ray of argument $0$ and the black dots are the marked points. On the right, the multicurve $\{a,b\}$ is drawn with $a$ in white and $b$ in black. On the left we drew the (accurate) preimage of the right by $F_R$. The white dots indicate those preimages of the black dots that are not themselves marked. Here, $R\approx 1.41$, $R^3\approx 2.78$.}
\label{fig:ccp}
\end{figure}
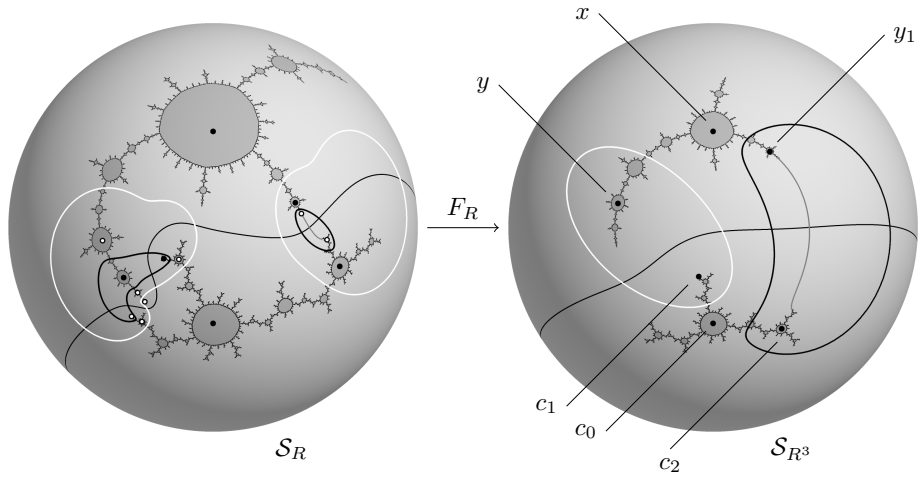

\begin{figure}%
\begin{tikzpicture}
\node at (0,0) {\includegraphics[width=7cm]{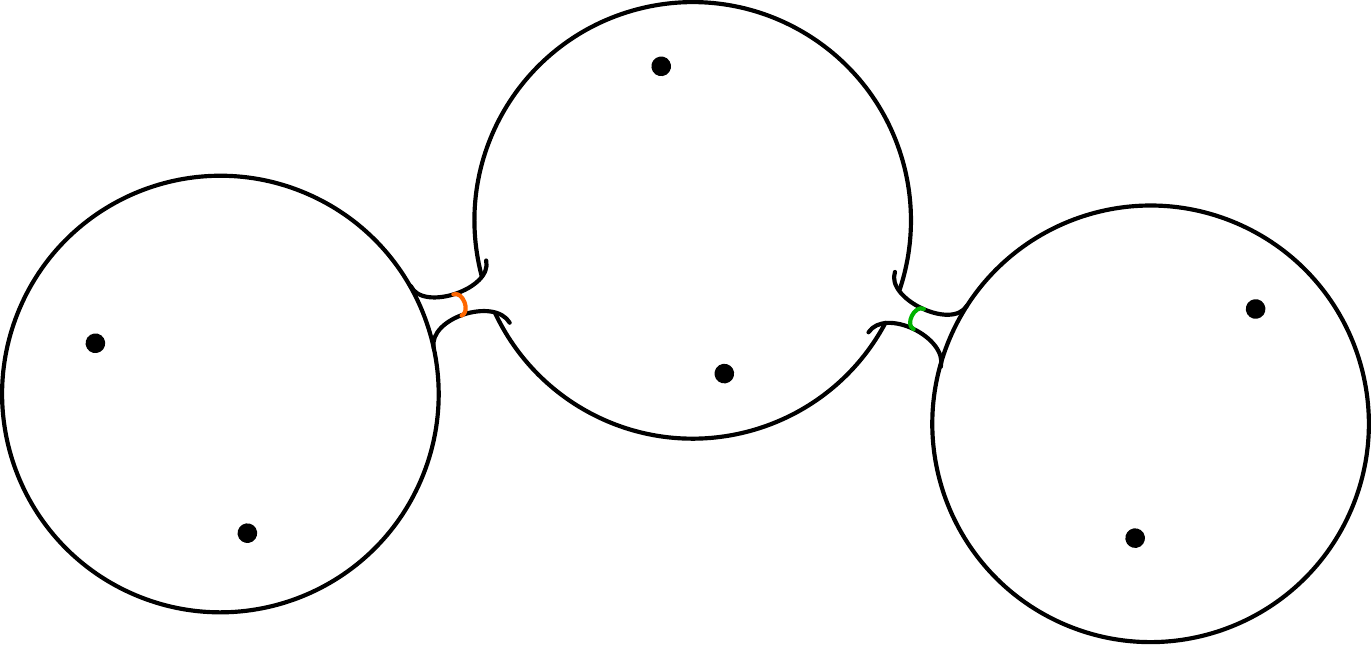}};
\node at (-0.1,-0.3) {$c_0$};
\node at (-1.9,-1) {$c_1$};
\node at (2,-1) {$c_2$};
\node at (-0.3,1.1) {$x$};
\node at (-2.8,0) {$y$};
\node at (2.7,0) {$y_1$};
\node at (-1.1,-0.3) {$a$};
\node at (1.1,-.3) {$b$};
\node at (-2.3,1.1) {$A_1$};
\node at (0,-1) {$A_2$};
\node at (2.4,1) {$A_3$};
\end{tikzpicture}
\caption{A pinched sphere model of $\cal S_R$ for $R$ close to $1$.}
\label{fig:H2O}
\end{figure}
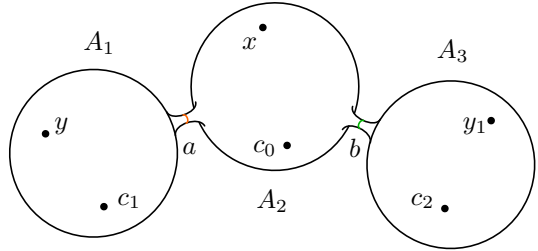

On Figure~\ref{fig:ccp} we drew the spheres $\cal S_R$ for a pair of values $R$ and $R^3$, together with the Julia sets and the equator. We also drew a multicurve\footnote{Multicurves are usually defined for Thurston maps $f:S\to S$ but in the proof of Thurston's theorem one also considers the image of multicurves by markings $\phi : S\to\wh{\C}$. See also Section~\ref{subsub:aboutpinch}.} corresponding to the obstruction $\Gamma=\{a,b\}$ and its preimage by $F_R$.

In Section~\ref{subsub:pinch} we proved that the canonical obstruction is $\Gamma=\{a,b\}$. Thus as $R\tend 1$, the marked sphere will be conformally equivalent to something looking like Figure~\ref{fig:H2O}.

\begin{figure}%
\scalebox{1.07}{%
\begin{tikzpicture}
\node at (0,0) {\includegraphics[width=11.5cm]{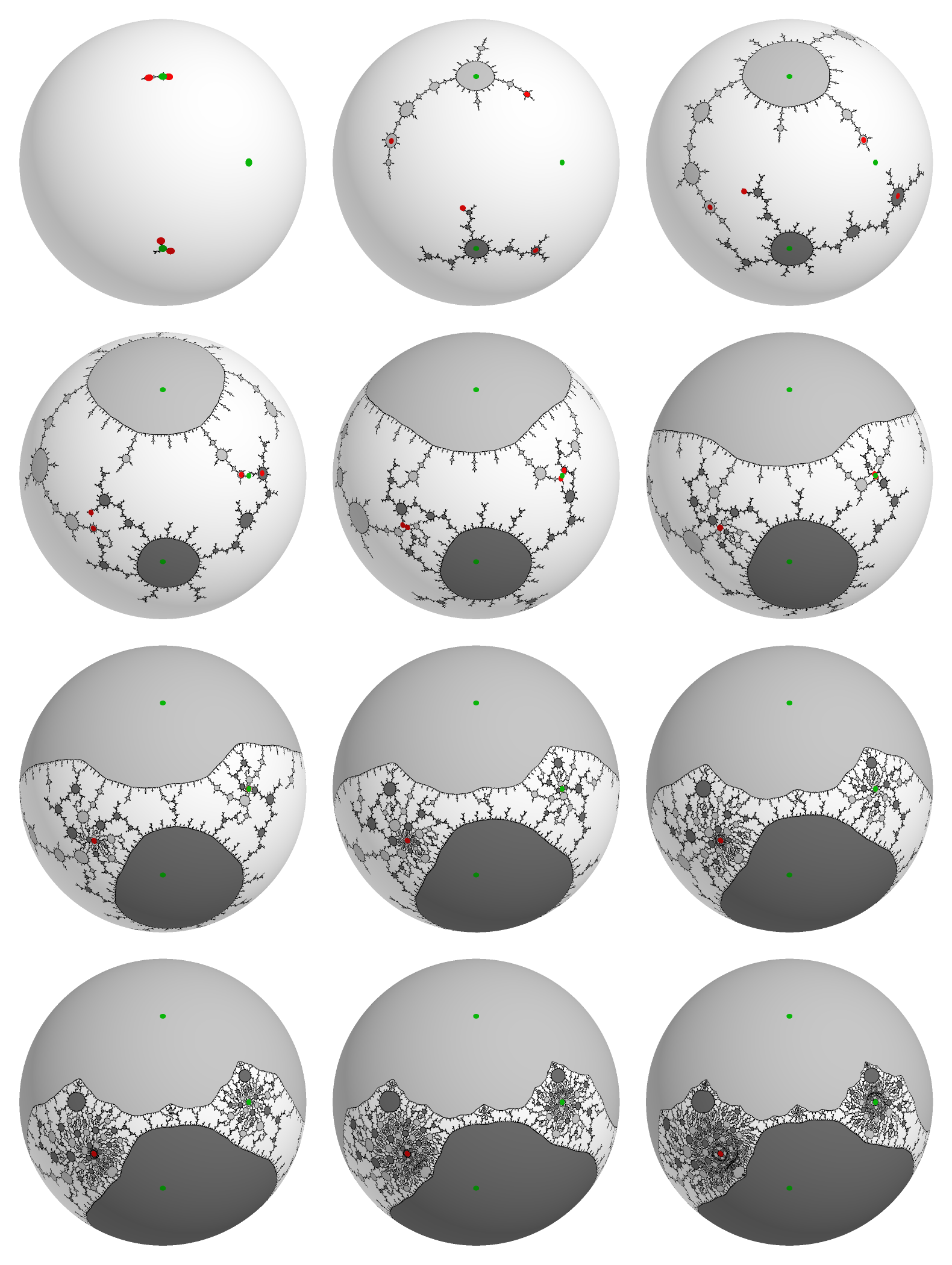}};
\node at (-3.8,4.25) {\textcolor{dkgreen}{$0$}};
\node at (-2.5,5.7) {\textcolor{dkgreen}{$1$}};
\node at (-3.8,7) {\textcolor{dkgreen}{$\infty$}};
\node at (0,4.2) {$c_0$};
\node at (1.4,5.6) {$e_0$};
\node at (0.9,4.2) {$c_2$};
\node at (-0.4,5.1) {$c_1$};
\node at (-0.3,7) {$x$};
\node at (-0.8,5.7) {$y$};
\node at (1,6.5) {$y_1$};
\draw (-5.4,3.8) -- (5.4,3.8);
\draw (-5.4,0) -- (5.4,0);
\draw (-5.4,-3.8) -- (5.4,-3.8);
\end{tikzpicture}
}
\caption{Sequence, centered on the middle sphere $A_2$. See the text for a description.}
\label{fig:movieA2}
\end{figure}
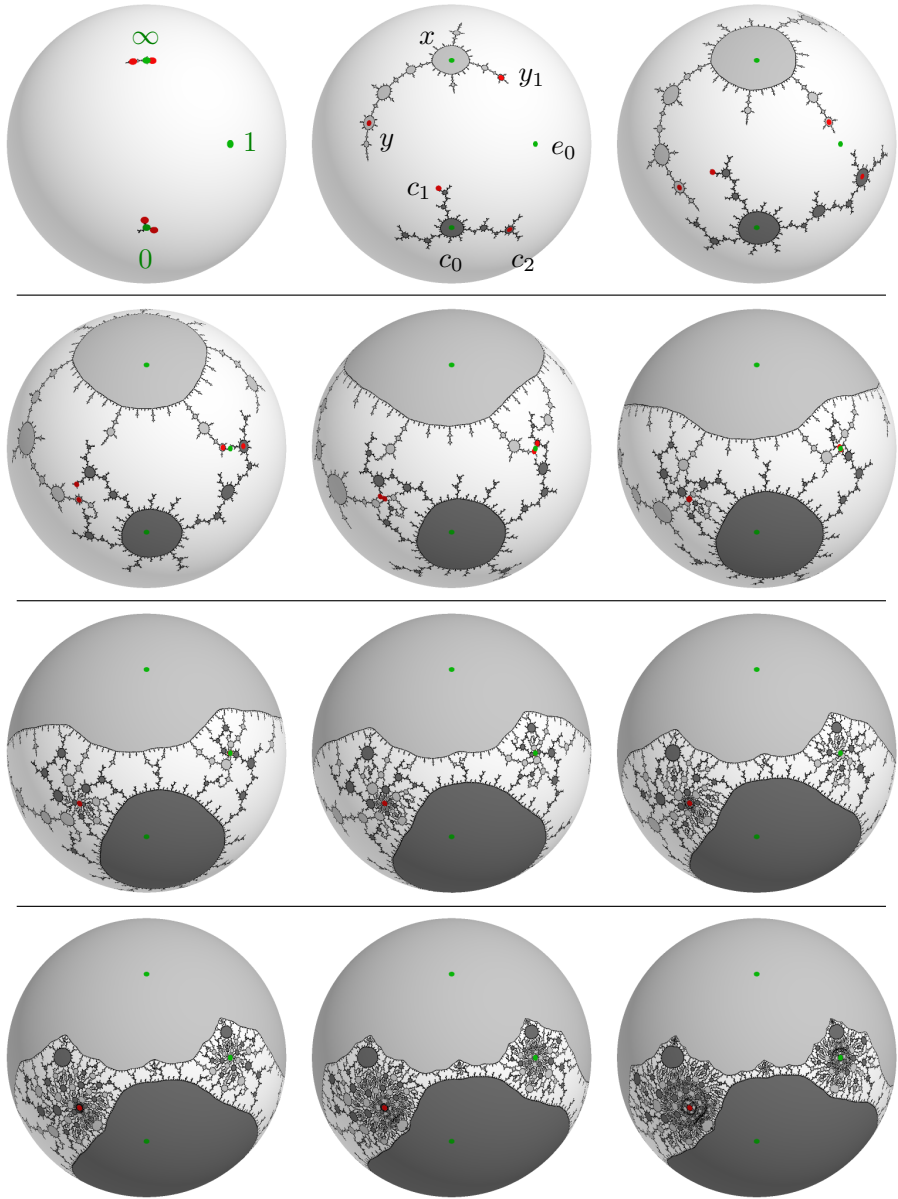

\begin{figure}%
\includegraphics[width=11.5cm]{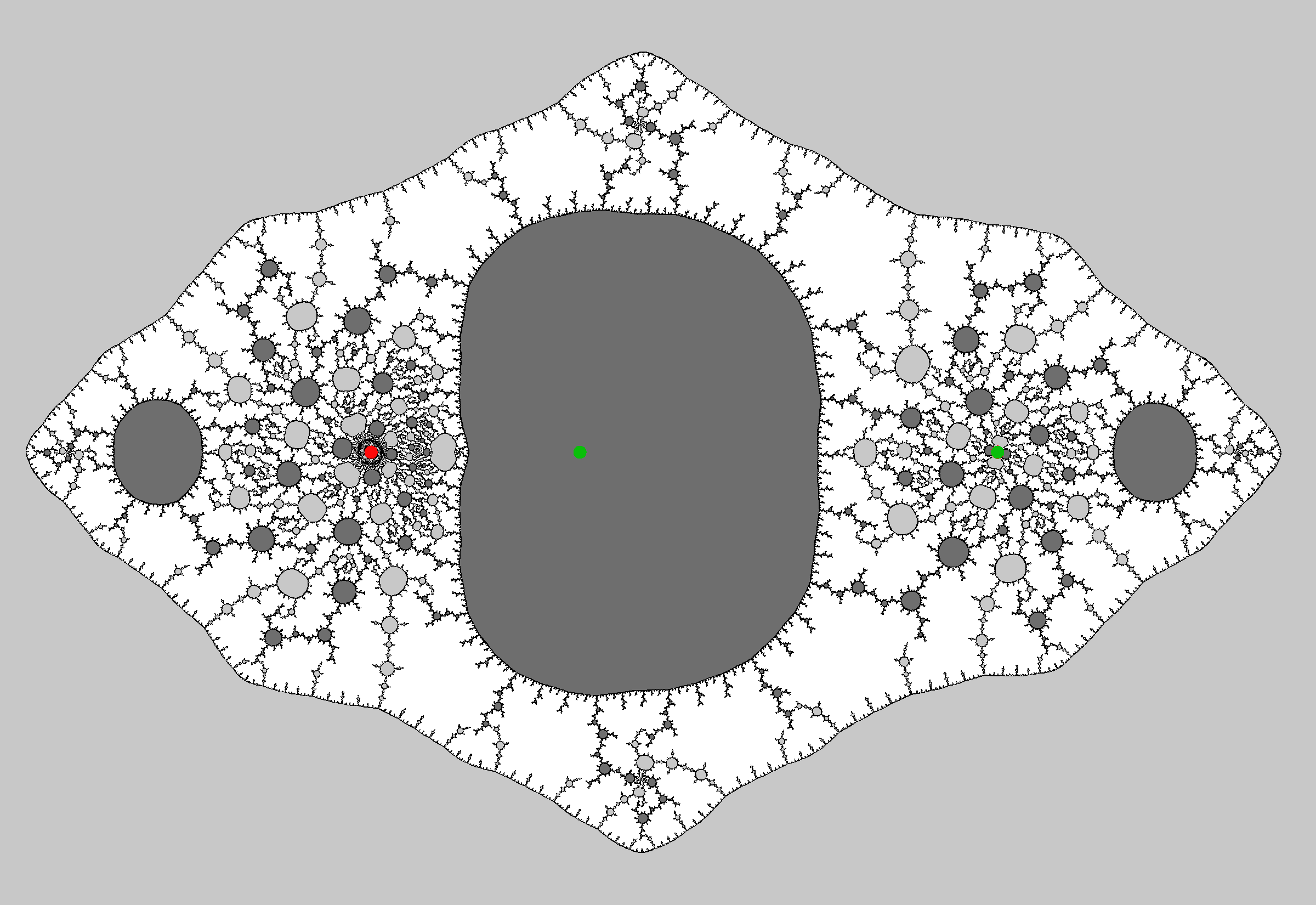}
\caption{Middle sphere centered flat view of $\cal S_R$ for $R=R_{10}$.}
\label{fig:10flat}
\end{figure}

\begin{figure}%
\includegraphics[width=11.5cm]{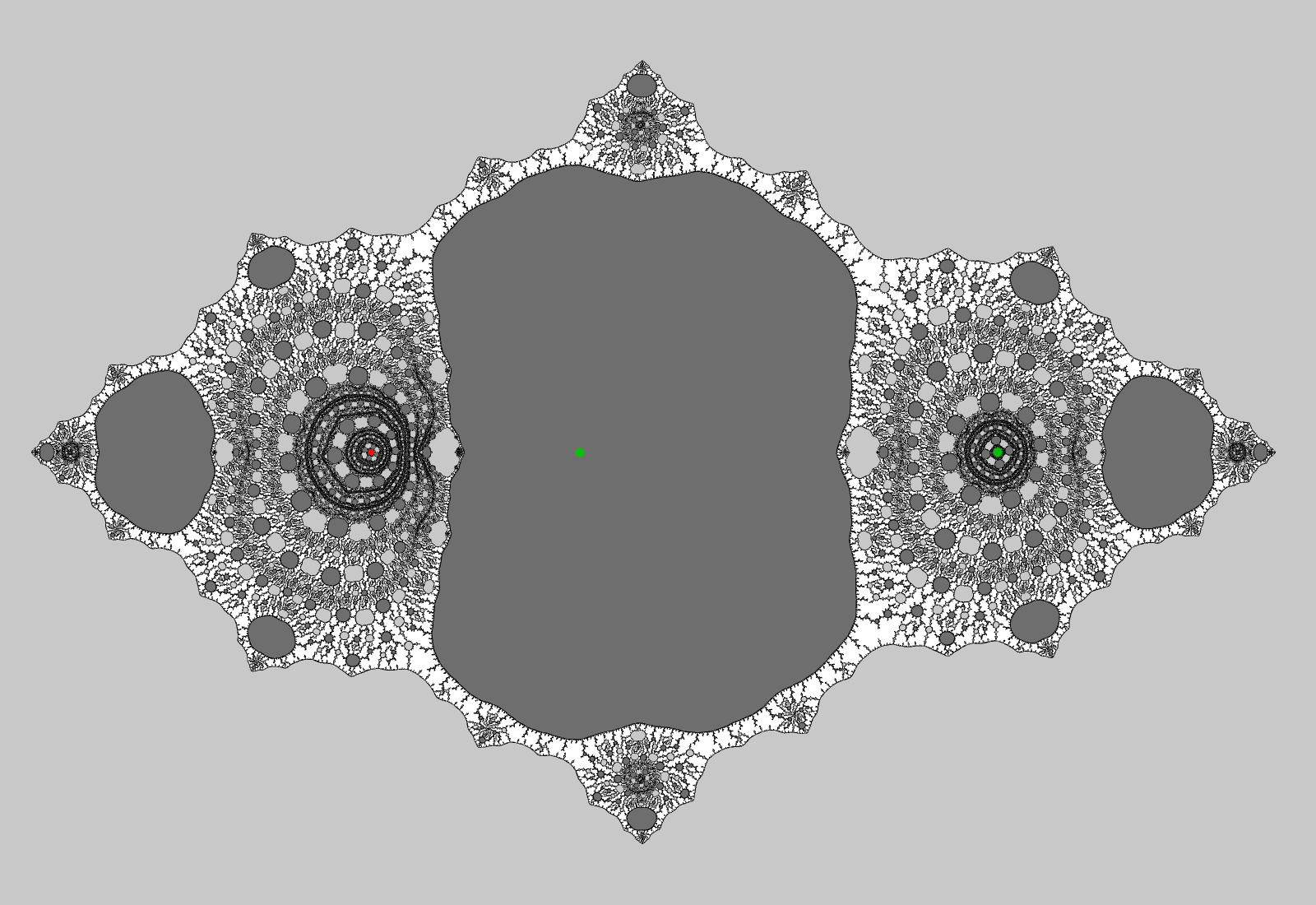}
\caption{Middle sphere centered flat view of $\cal S_R$ for $R=R_{15}$.}
\label{fig:15flat}
\end{figure}

\begin{figure}%
\includegraphics[width=11.5cm]{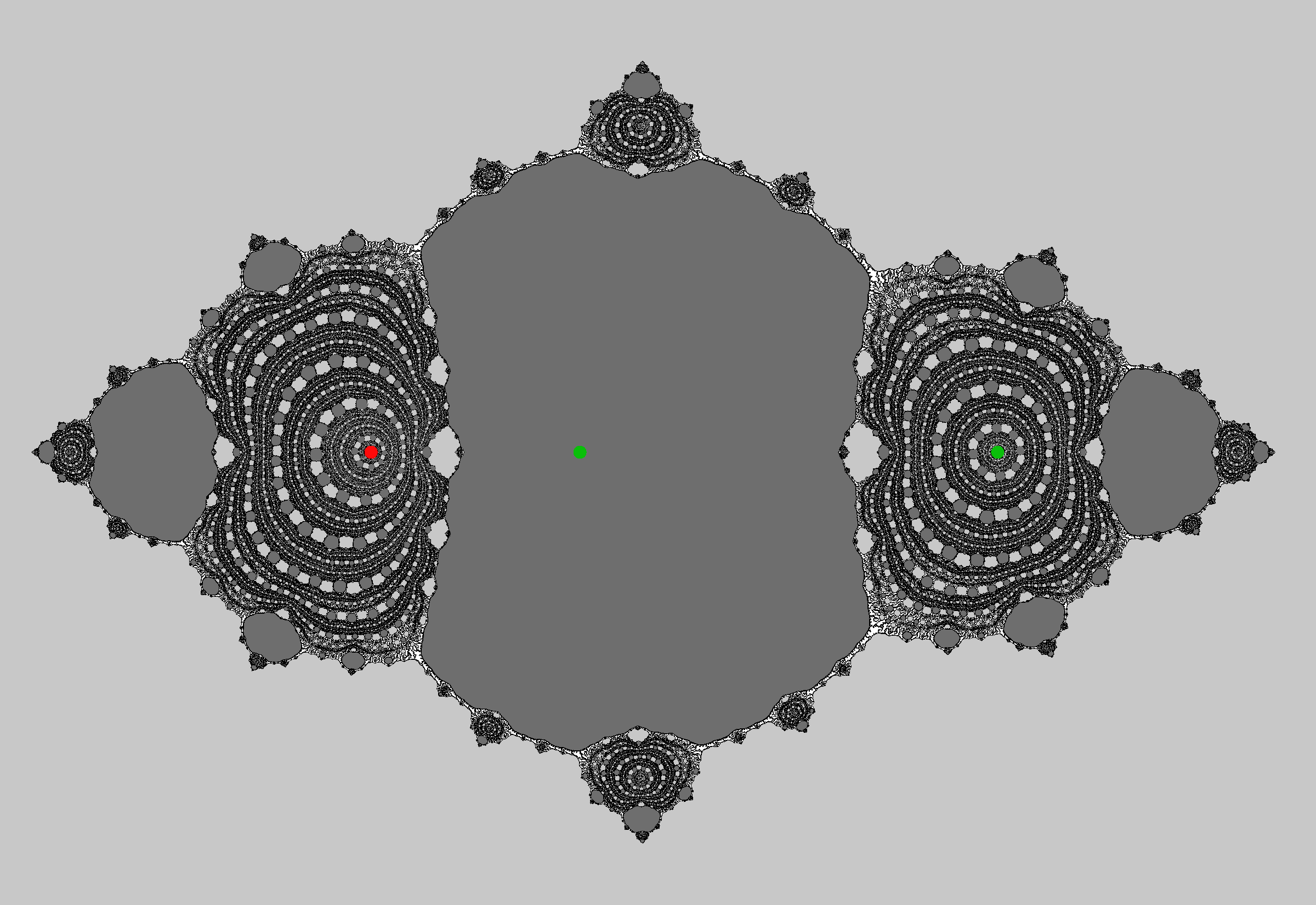}
\caption{Middle sphere centered flat view of $\cal S_R$ for $R=R_{20}$.}
\label{fig:20flat}
\end{figure}

\begin{figure}%
\scalebox{1.1}{%
\begin{tikzpicture}
\node at (0,0) {\includegraphics[width=11.5cm]{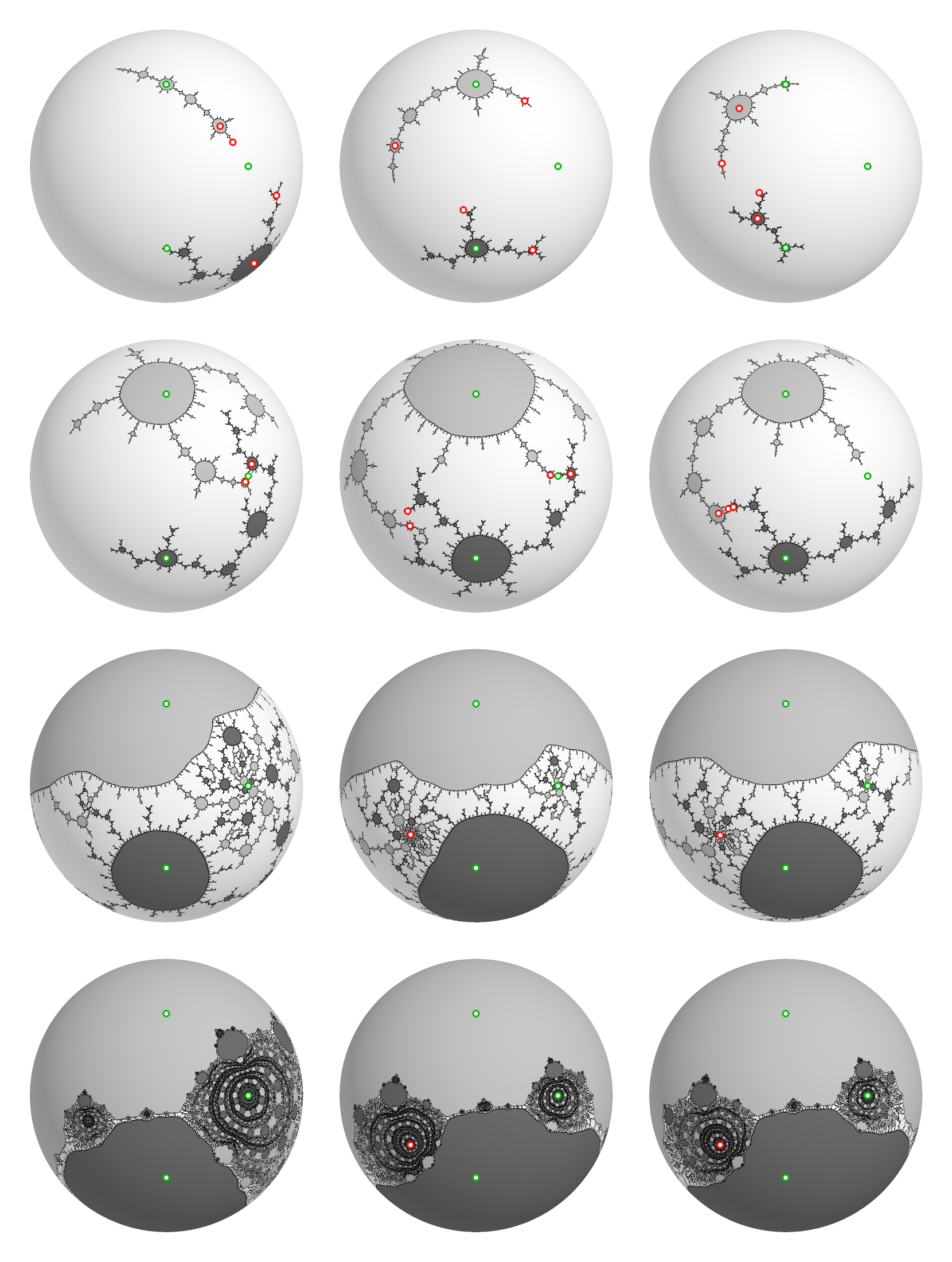}};
\draw (-1.9,-7.2) -- (-1.9,7.2);
\draw (1.9,-7.2) -- (1.9,7.2);
\end{tikzpicture}
}
\caption{Each row shows the same sphere $\cal S_R$ for $R=$ respectively $R_2$, $R_4$, $R_8$ and $R_{16}$, with
 $R_n=10^{4/3^n}$. Each column has a different normalization: from left to right, $(c_1,e,y)$, $(c_0,e,x)$ and $(c_2,e,y_1)$ are mapped to $(0,1,\infty)$.}
\label{fig:parallelMovie}
\end{figure}

On Figure~\ref{fig:movieA2} we show a sequence of twelve spheres $\cal S_{R_1}$ to $\cal S_{R_{12}}$ with $R_n = 10^{4/3^{n}}$, which satisfies $R_{n+1}^3=R_n$, so that each sphere maps to the previous by $F_{R_{n+1}}$; it ranges from $R_1=21.54\ldots$ to $R_{12}=1.0000173\ldots$ The marked points and the point on the equator of argument $0$, call it $e_0$, are highlighted as red and green dots according to a convention that we explain now: the green dots indicate the points used for the normalization of the sphere, and the red dots indicate those marked points that are not green. The normalization was chosen as follows: $c_0$ is sent to $0\in\wh\C$; $x$ is mapped to $\infty\in\wh\C$; the point $e_0$ is mapped to $1\in\wh\C$. Hence two marked points are green. The conformal projection from $\wh\C$ to the Euclidean sphere has been chosen so that the three green points are all visible, with $0$ near the bottom, $1$ on the right and $\infty$ near the top.

This choice of normalization is different from the ones proposed in Section~\ref{subsub:aboutpinch}, which consist in putting at $0$, $1$ and $\infty$ three marked points, whereas here we only put two plus a point on the equator. The reason is that for values of $R$ that are not small, this gives a more balanced picture: indeed, if the three marked were normalized, then at least two of them would belong to one of the two polynomial. As a consequence, on the picture, only that polynomial would be visible. The drawback is that it is not obvious on which piece of the sphere cut along the short geodesics we are zooming in.

On the sequence on Figure~\ref{fig:movieA2}, it appears that we are zooming on the middle sphere, $A_2$: the points $y$ and $c_1$ seem to gather together, like do the points $y_1$, $c_2$ and $e_0$. The explanation would be that in the normalizations focusing on $A_3$ (sending the triple $y_1$, $c_2$ and $z=$ any of the remaining four marked points to $\infty$,$0$,$1$) $e_0$ converges to a point distinct from the limit of the marked points (i.e.\ from $0$, $1$ and $\infty$).

\begin{question}Prove this assertion.\footnote{For instance $e_0$ could be added to the set of marked points and Thurston's algorithm extended to include cycles. See the work of Selinger in \cite{Se}.}
\end{question}

On Figures~\ref{fig:10flat}, \ref{fig:15flat} and \ref{fig:20flat} we drew (still conformally correct) flat views, where the green point $x$ is sent to infinity, for $R_{10}$,  $R_{15}$ and $R_{20}$.

On Figure~\ref{fig:parallelMovie} we show in parallel three sequences of four vertically stacked images, each sequence has a different normalization, each row has the same value of $R$. Each normalization focuses on a different piece: $A_1$, $A_2$, $A_3$. Notice how the middle and right column converge to a similar limit.

\subsection{Interpretation of the pictures.}

In the three normalizations focusing on respectively $A_1$, $A_2$ and $A_3$, even though the family of maps $F_R$ is supposed to diverge, the picture seems still to converge to a well defined limit, very reminiscent of a Julia set, but with mess in some Fatou components. We propose the following (conjectural) interpretation. Work of Selinger in \cite{Se} may help in giving proofs.

As $R\tend 1$, the marked sphere tends to a bunch of three Riemann spheres: left, middle and right, limits of respectively $A_1$, $A_2$ and $A_3$, touching each other at two singular points. The map $F_R$ has a limit $\cal F$ on this bunch minus the touching points, which maps the three spheres between themselves, as rational maps.
More precisely the limit $\cal F$ sends the middle sphere to the left, the left to the right, and the right back to the middle. Let us write $\cal F_1$, $\cal F_2$ and $\cal F_3$ for the three rational maps between these spheres (the domains of definition of the $\cal F_i$ include the touching points but the maps will not necessarily agree on these points, and the convergence will occur on the spheres minus these points).

On the middle sphere, the limit $\cal F_2$ has a critical point at $x$, and a double critical point at $c_0$. There cannot be an odd number of critical points counted with multiplicity so there must be a supplementary one. It is hiding at the touching point between the left and middle spheres. Indeed, recall that the multicurve component $a$ has a pre-image component which is homotopic to $a$ and which is mapped $2:1$ to $a$ (in an orientation reversing way). There is no critical point at the singular point between the middle and the right spheres, because the component of the preimage of $a$ that is homotopic to $b$ is sent $1:1$ to $a$.

The same analysis on the left sphere indicates that $\cal F_1$ has two critical points, so it has only degree $2$. The critical points are $y$ and the touching point, and they respectively map to $y_1$ and the other touching point. Finally, the map $\cal F_3$ has no critical point. Thus the respective degrees of $\cal F_1$, $\cal F_2$, $\cal F_3$, are 2, 3, 1. Another way of seeing that would have been to look at the preimage of the nearly pinched sphere by $F_R$: this gives a pinched sphere with five bubbles, as on the second row of Figure~\ref{fig:tent}, three of which contain the marked points (thus excluding $e_0$, which is not marked).  Counting degrees then gives the same result.

Knowing this and the way marked points and touching points are mapped to each other, is enough, at least on the example we are studying, to completely determine the maps $\cal F_1$,  $\cal F_2$, $\cal F_3$. Let us stick to the normalization used in the sequence of pictures previously shown: before the limit, map $e_0$ to $1$ on each sphere, and the two non collapsing marked point in each sphere to $0$ and $\infty$ for respectively $P_2$ and $P_1$. This means on the limit: to $1$ is mapped the touching point on the left sphere, the right touching point on the middle sphere, and the point $e_0$ on the right sphere; to $0$ and $\infty$ are mapped the two marked points in each sphere. Then the three rational maps $\cal F_i$ must send $0$ to $0$, $1$ to $1$ and $\infty$ to $\infty$. Since moreover $\cal F_3$ has degree $1$, it is the identity:
\[\cal F_3(z)=z.\]
Note that $\cal F_1$ is a polynomial and that the conjugate $G_2$ of $\cal F_2$ by $z\mapsto 1/z$ is polynomial (recall that a rational map is polynomial if and only if the preimage of $\infty$ is itself, no more, no less). 
The degree of $G_2$ is $3$ and $G_2'(0)=0$, $G_2(0)=0$, so $G(z)=bz^3+az^2$ with $b\neq 0$. Also
$G_2(1)=1$ so $a+b=1$, and $G_2$ maps the remaining critical point $c=-2a/3b$ (whose coordinate in the middle sphere we do not know yet) to $1$. This gives $4a^3=27b^2$. So $b=1-a$ and $4(a/3)^3-a^2+2a-1=0$. The latter equation has a double root $a=3$ and a single root $a=3/4$. In the case $a=3$, we get that the other critical point $c=1$, which we exclude: the two touching points cannot be equal. Thus $a=3/4$, $b=1/4$, $c=-2$ and $G_2(z)=z^2(z+3)/4$:
\[\cal F_2(z)=\frac{4z^3}{3z+1}.\]
We also got the coordinates of the left touching point on the middle sphere, which is a critical point of $\cal F_2$ and is $1/c=-1/2$.
Finally, $\cal F_1$ is a degree $2$ polynomial with a critical point at $z=1$, fixing $0$ and mapping $1$ to the touching point on the right sphere. The latter being mapped to the touching point $-1/2$ on the middle sphere by $\cal F_3=\on{id}$ we get that $\cal F_1(1)=-1/2$.
\[\cal F_1(z)=((z-1)^2-1)/2=z^2/2-z.\]
This is summarized on Figure~\ref{fig:summary}.

\begin{question} Prove that the limit of $F_R$ in the pinching spheres representation is indeed the map $\cal F = \cal F_1 \cup \cal F_2 \cup \cal F_3$.\end{question}

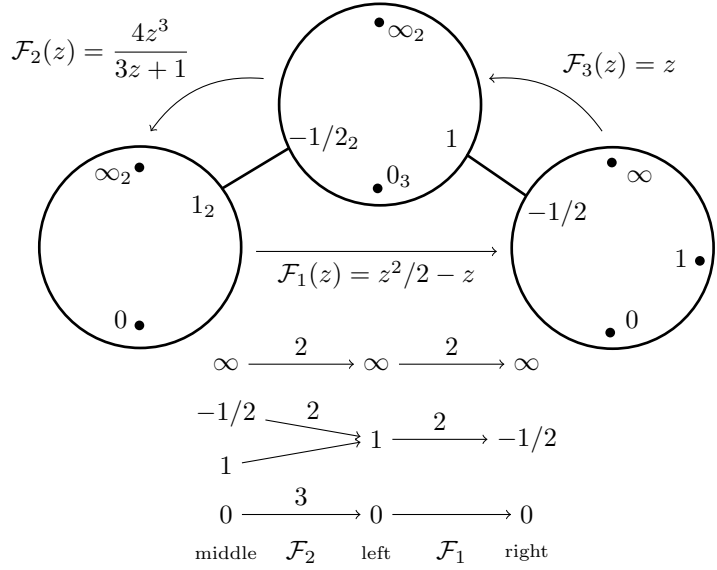
\begin{figure}%
\begin{tikzpicture}
\node at (0,0) {\includegraphics[width=9cm]{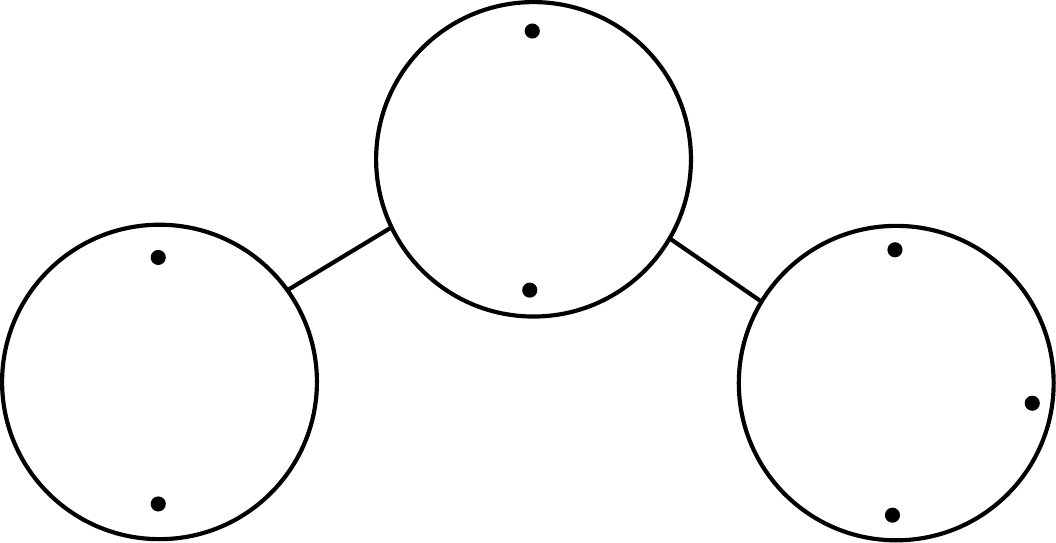}};

\node at (-3.5,0) {$\infty_2$};
\node at (.4,1.9) {$\infty_2$};
\node at (3.5,0) {$\infty$};
\node at (-3.4,-1.9) {$0$};
\node at (0.3,0) {$0_3$};
\node at (3.4,-1.9) {$0$};
\node at (4.05,-1.1) {$1$};
\node at (2.4,-0.5) {$-1/2$};
\node at (1,0.5) {$1$};
\node at (-0.7,0.5) {$-1/2_2$};
\node at (-2.3,-0.4) {$1_2$};
\draw [->] (3,0.6) to [bend right] node [above right] {$\cal F_3(z) = z$} (1.5,1.3);
\draw [->] (-1.5,1.3) to [bend right] node [above left] {$\ds\cal F_2(z) = \frac{4z^3}{3z+1}$} (-3,0.6);
\draw [->] (-1.6,-1) to node [below] {$\cal F_1(z) = z^2/2-z$} (1.6,-1);

\node (i1) at (-2,-2.5) {$\infty$};
\node (i2) at (0,-2.5) {$\infty$};
\node (i3) at (2,-2.5) {$\infty$};
\draw [->] (i1) -- node [above] {$2$} (i2);
\draw [->] (i2) -- node [above] {$2$} (i3);
\node (i1) at (-2,-3.15) {$-1/2$};
\node (i0) at (-2,-3.85) {$1$};
\node (i2) at (0,-3.5) {$1$};
\node (i3) at (2,-3.5) {$-1/2$};
\draw [->] (i1) -- node [above] {$2$} (i2);
\draw [->] (i0) -- (i2);
\draw [->] (i2) -- node [above] {$2$} (i3);
\node (i1) at (-2,-4.5) {$0$};
\node (i2) at (0,-4.5) {$0$};
\node (i3) at (2,-4.5) {$0$};
\draw [->] (i1) -- node [above] {$3$} (i2);
\draw [->] (i2) --(i3);
\node at (-1,-5) {$\cal F_2$};
\node at (1,-5) {$\cal F_1$};
\node at (-2,-5) {\scriptsize middle};
\node at (0,-5) {\scriptsize left};
\node at (2,-5) {\scriptsize right};
\end{tikzpicture}
\caption{Summary. Above: the subscript indicates the multiplicity of the critical points. Cross ratio of complex numbers is not respected on this schematic illustration. Below: who maps where.}
\label{fig:summary}
\end{figure}

Now, the third iterate of the map $\cal F$ is a map from the bunch to itself which sends each sphere to itself, by three degree $6$ maps that are semi-conjugate to each other because they all consist in making one turn in the following non-commutative diagram:

\centerline{%
\begin{tikzpicture}
[my style/.style={circle,draw=black,minimum size=6mm,outer sep=2pt}]
\node[my style] (a) at (0,1.4) {};
\node[my style] (b) at (-.85,0) {};
\node[my style] (c) at (.85,0) {};
\path[->] (a) edge node[above left]{$\cal F_2$} (b)
	(b) edge node[below]{$\cal F_1$} (c)
	(c) edge node[above right]{$\cal F_3$} (a);
\end{tikzpicture}%
}

The map in the middle is $H_2= \cal F_3\circ \cal F_1\circ \cal F_2$, the map on the left is $H_1=\cal F_2 \circ \cal F_3\circ \cal F_1$ and the map on the right is $H_3=\cal F_1\circ \cal F_2\circ \cal F_3$. Since $\cal F_3$ is the identity in this example and with our choices of coordinates, we only get two different degree $6$ rational maps. Let us characterize further these two maps.
They both have three critical values\footnote{In this respect they are rigid in the following sense: any rational map topologically equivalent to one of them as a ramified cover, by a pair of orientation preserving homeomorphisms, must me equivalent to it by a pair of homographies.}: $0$, $1$, $\infty$. The map $H_2$ from the middle sphere to itself has $5$ critical points: $\infty \overset{4}\longrightarrow \infty$,  $1\overset{4}\longrightarrow 1$, $-1/2\overset{2}\longrightarrow -1/2$, $-1/3\overset{2}\longrightarrow \infty$ and $0 \overset{3}\longrightarrow 0$, the number above the arrow indicates the local degree. The sum of ``local degrees minus one'' is $10$, which agrees with the formula $2d-2$ counting the number of critical points with multiplicity for a degree $d=6$ map.
The map $H_1$ on the left sphere has $4$ critical points: $\infty \overset{4}\longrightarrow \infty$,  $1 \overset{4}\longrightarrow 1$,  $0 \overset{3}\longrightarrow 0$,  $2 \overset{3}\longrightarrow \infty$.
This is better presented as a diagram:

\centerline{%
\begin{tikzpicture}
\node (z) at (5,0) {$0$};
\node (i) at (5,-.8) {$\infty$};
\node (u) at (5,-1.6) {$1$};
\node (mz) at (3.5,0) {$2$};
\path[->] (mz) edge node[above]{$\scriptstyle 3$} (z)
                 (z) edge [loop right] node{$\scriptstyle 3$} ();
\path[->] (i) edge [loop right] node{$\scriptstyle 4$} ();
\path[->] (u) edge [loop right] node{$\scriptstyle 4$} ();
\node[left] at (6,-2.4) {$\cal F_2\circ \cal F_3 \circ \cal F_1$};
\begin{scope}[xshift=10cm]
\node (z) at (0,0) {$0$};
\node (i) at (0,-.8) {$\infty$};
\node (u) at (0,-1.6) {$1$};
\node (mi) at (-1.5,-.8) {$-1/3$};
\node (mu) at (-1.5,-1.6) {$-1/2$};
\path[->] (z) edge [loop right] node{$\scriptstyle 3$} ();
\path[->] (mi) edge node[above]{$\scriptstyle 2$} (i) 
	         (i) edge [loop right] node{$\scriptstyle 4$} ();
\path[->] (mu) edge node[above]{$\scriptstyle 2$} (u)
                 (u) edge [loop right] node{$\scriptstyle 4$} ();
\node[left] at (1,-2.4) {$\cal F_3\circ \cal F_1 \circ \cal F_2\phantom{)}$};
\node[left] at (1,-2.8) {$(=\cal F_1\circ \cal F_2 \circ \cal F_3)$};
\end{scope}
\end{tikzpicture}%
}

\begin{figure}
\begin{tikzpicture}

\node at (0,0) {\includegraphics[width=11.5cm]{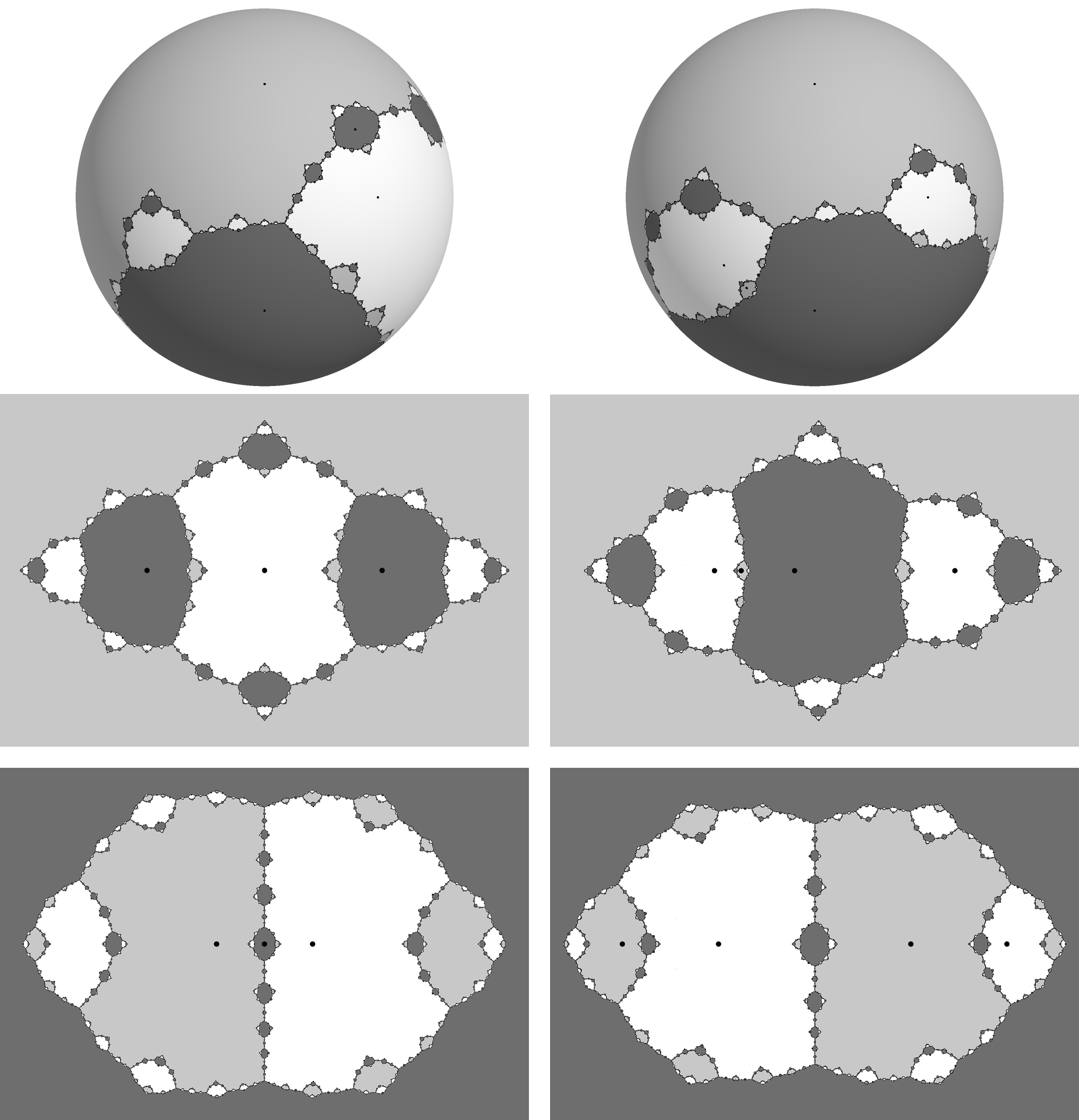}};

\node at (-2.8,5.3) {$\infty$};
\node at (-2.8,2.8) {$0$};
\node at (-1.6,4) {$1$};
\draw (-1.9,4.65) -- (-1,5.3) node[right] {$2$};

\node at (3.1,5.3) {$\infty$};
\draw (1.9,3.1) -- (1,2.5) node [below left] {$\scriptstyle -1/2$};
\draw (2.2,2.87) --(1.7,2.15) node [below left] {$\scriptstyle -1/3$};
\node at (3.1,2.8) {$0$};
\node at (4.3,3.85) {$1$};

\node at (-3.2,4.8) {*};
\node at (-2.2,3.8) {*};
\node at (-3.2,3) {*};

\node at (2.7,4.6) {*};
\node at (2.2,3.5) {*};
\node at (2.9,3.3) {*};

\node at (-4.3,0.1) {$0$};
\node at (-2.8,0.1) {$1$};
\node at (-1.6,0.1) {$2$};

\node at (1.7,0.1) {$\scriptstyle -1/2$};
\node at (2.1,-0.4) {$\scriptstyle -1/3$};
\node at (2.85,0.1) {$0$};
\node at (4.3,0.1) {$1$};

\path[->] (-1.5,-1.05) edge [bend right=30] node [below] {$z\mapsto z^2/2-z$} (1.5,-1.05);
\path[->] (1.5,0.85) edge [bend right=30] node [below] {$z\mapsto \frac{4z^3}{3z+1}$} (-1.5,0.85);

\node at (-3.6,-3.9) {$\infty$};
\node at (-2.8,-3.9) {$2$};
\node at (-2.2,-3.9) {$1$};

\node at (0.8,-3.8) {$-1/3$};
\node at (1.8,-3.8) {$-1/2$};
\node at (4,-3.8) {$\infty$};
\node at (5,-3.8) {$1$};

\end{tikzpicture}
\caption{Julia sets of the maps $H_i$ (third iterate of the presumed limit on the tree of spheres). The three basins of attraction are drawn in white and shades of gray. We have put a mark ``$*$'' in each immediate basin in the first row. Compare with Figure~\ref{fig:parallelMovie}. The column on the left represents $H_1$ (left sphere). The column on the right represents $H_2$ and $H_3$ (middle and right spheres). The first row is the sphere view, the second row is the flat view, and the last is the flat view followed by an inversion putting $0$ at infinity. Putting the third fixed critical point at infinity would give an image identical to the second row but with a permutation of light and dark gray.}
\label{fig:lims}
\end{figure}
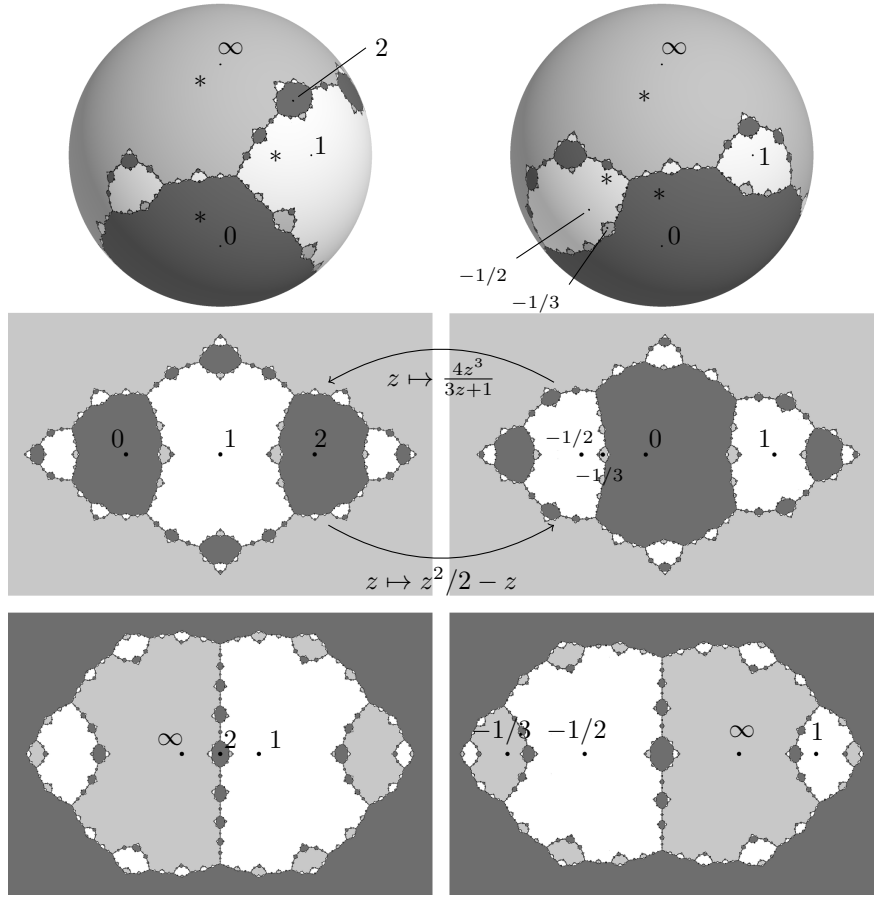

On Figure~\ref{fig:lims} are shown pictures of the Julia and Fatou sets of these two maps $H_1$ and $H_2$. For both, the behavior of the critical points and the classification of Fatou components imply that every Fatou component eventually falls under iteration in the immediate basin of the three fixed critical point of $H_i$.
The picture suggests that there may be better coordinates in which to express them. For instance if one takes the coordinates $w=z-1$ in the left sphere and $u=1/(2z+1)$ in the middle sphere, then we get
\[\cal F_3\circ \cal F_1 : w\mapsto u=1/w^2\text{ and }\cal F_2 : u \mapsto w=-\frac{1}{u^2}\cdot\frac{u-1/3}{1-u/3}.\]
These are two real Blaschke fractions. So they commute with $z\mapsto \ov{z}$ and $z\mapsto 1/z$ and preserve both the real axis and the unit circle. Their compositions $H_1$ and $H_2$ have relatively simple expressions: $H_2 =\cal F_3 \circ \cal F_1 \circ \cal F_2$ reads
\[u \mapsto u^4\left(\frac{1-u/3}{u-1/3}\right)^2\]
and $H_1 =\cal F_2 \circ \cal F_3 \circ \cal F_1$ reads
\[ w \mapsto -w^4\frac{1-w^2/3}{w^2-1/3}.\]
Another presentation would take $s=1-2/z$ in the left sphere and $t=1+1/z$ in the middle sphere, and then
\[\cal F_3\circ \cal F_1 : s\mapsto t=\frac{s+s^{-1}}{2}\text{ and }\cal F_2 : t \mapsto s=\frac{3t-t^3}{2}.\]
These are odd real maps, so they commute with $z\mapsto \ov{z}$ and $z\mapsto -z$ and preserve the real and the imaginary axis. The composition $H_1 = \cal F_2 \circ \cal F_3 \circ \cal F_1$ reads \[s\mapsto \frac{-1}{8}\left(s^3-3s-\frac{3}{s}+\frac{1}{s^3}\right).\]

\begin{figure}%
\includegraphics[width=11.5cm]{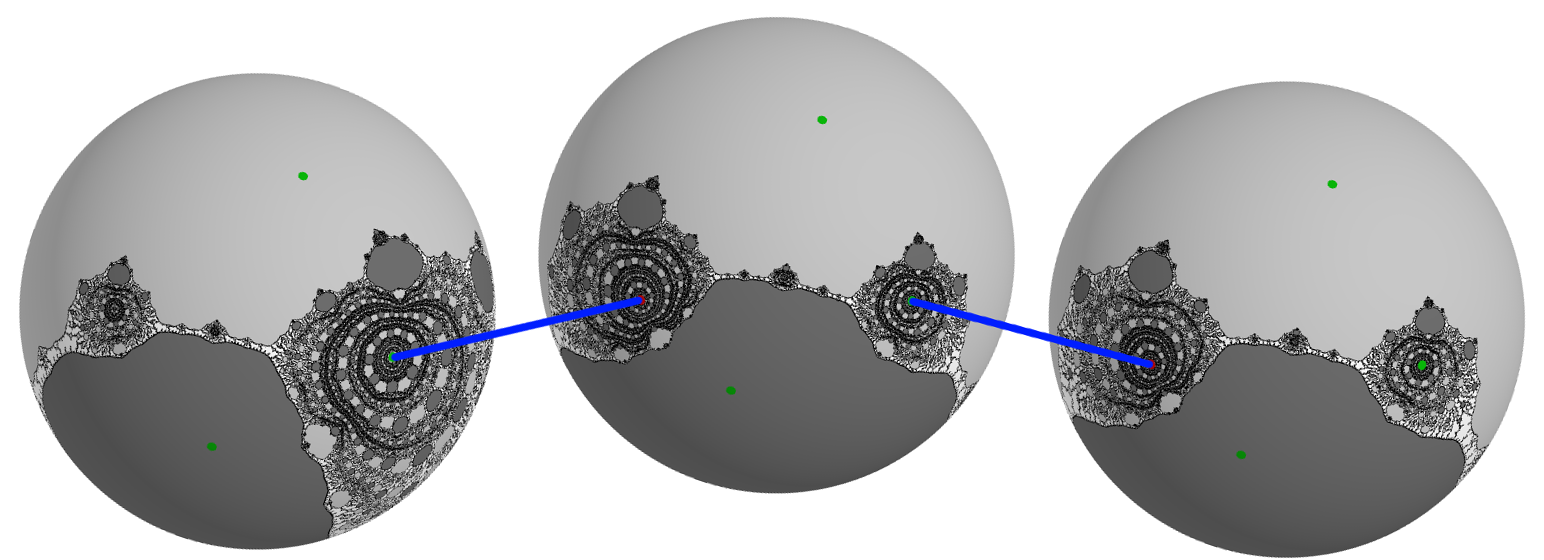}
\caption{The three spheres at the bottom of Figure~\ref{fig:parallelMovie}, with tubes between them symbolized by blue lines.}
\label{fig:pm2}
\end{figure}

Let us show again the bottom row of Figure~\ref{fig:parallelMovie}, and add blue lines symbolizing the tunnels between the three spheres (one can cut little openings in each sphere and add tunnels to glue them together, and recover the inital Riemann sphere $\cal S_R$): this gives Figure~\ref{fig:pm2}.
Comparing Figures~\ref{fig:pm2} and~\ref{fig:lims}, the reader will have noticed that one of the basins of the limit maps $H_i$ is replaced by some sort of mess, that gets denser and denser as $R\to 1$. In any of the three normalizations, the Hausdorff limit of the two Julia sets seems to contain the whole basin that is white in Figure~\ref{fig:lims}.

\begin{question} Prove that the Hausdorff limit of the Julia set in these three normalizations is indeed the union of the Julia set of $H_i$ and of the basin of attraction that is white in Figure~\ref{fig:lims}.\end{question}

In the next section we will focus on what happens in the tunnels between the three spheres. This may help to prove the assertion above.

\subsection{Zooming on a tunnel.}

\begin{figure}%
\includegraphics[width=10cm]{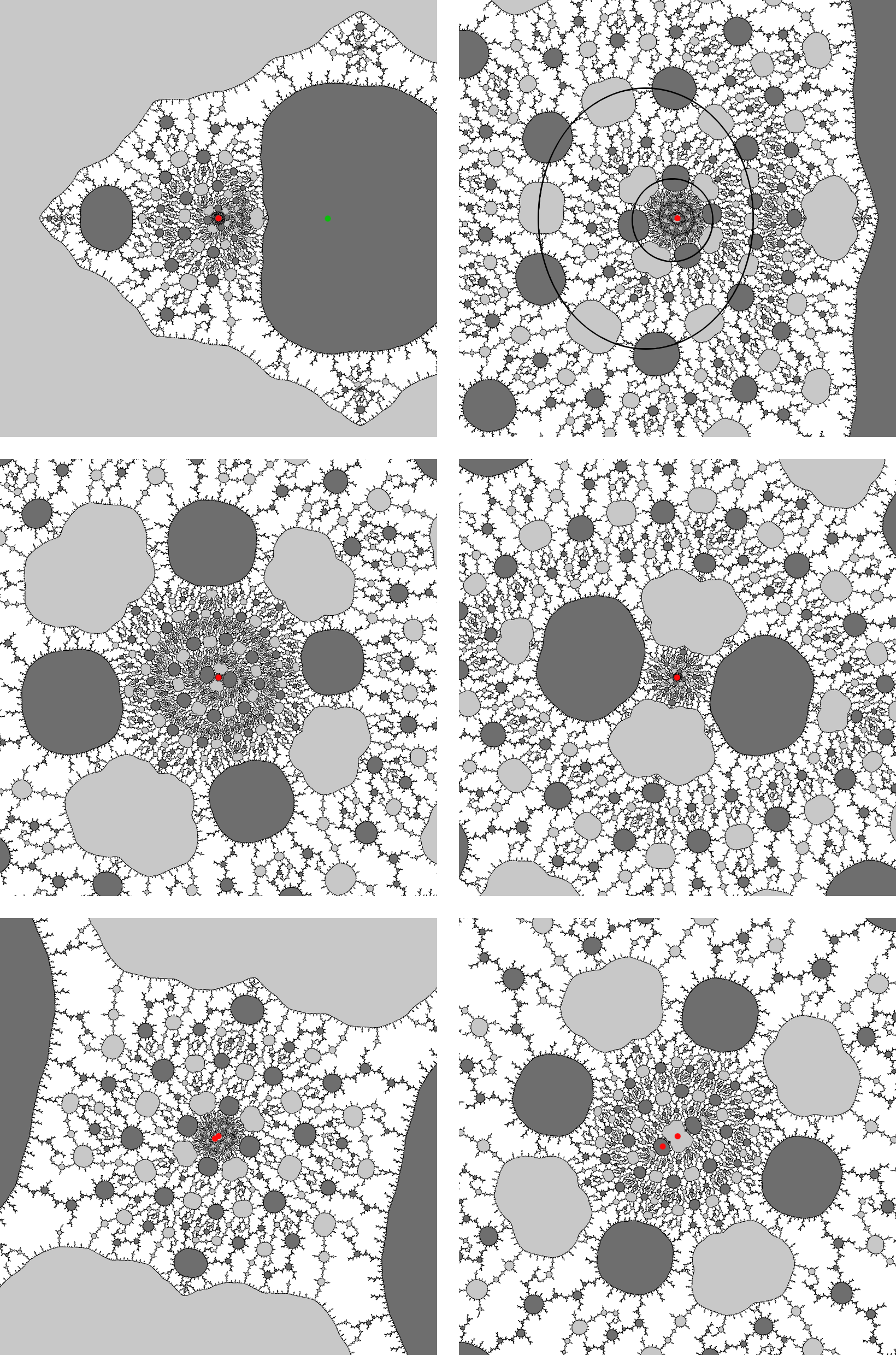}
\caption{A zoom on the messy part. This is a traversal of the tube starting from the middle sphere and ending at the left sphere. In the center of the last picture, we can see two red dots, which are the two marked points in the left sphere. On one picture, we also drew loops linking chains of light and dark gray components, which are approximately iterated square roots of the pictures in the middle sphere for bigger values of R.}
\label{fig:copies}
\end{figure}

\begin{figure}%
\includegraphics[width=9cm]{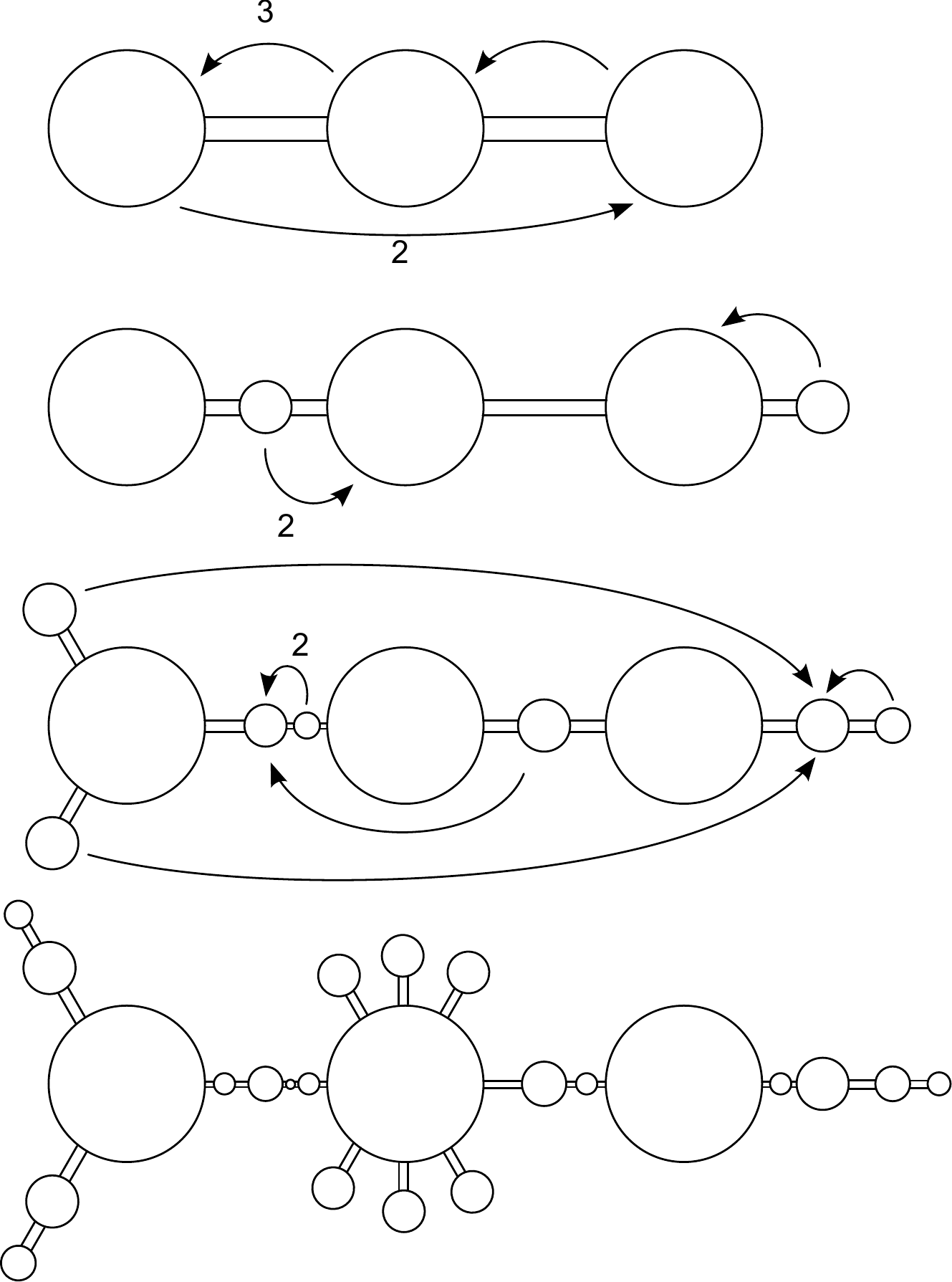}%
\caption{Pulling back the (partially) pinched sphere tree by the maps $\cal F_R$ defines more and more pinchings. The tube between the left and the middle initial spheres gets more and more spheres, some of which separate the two inital spheres.}
\label{fig:tent}
\end{figure}

\begin{figure}%
\includegraphics[width=10cm]{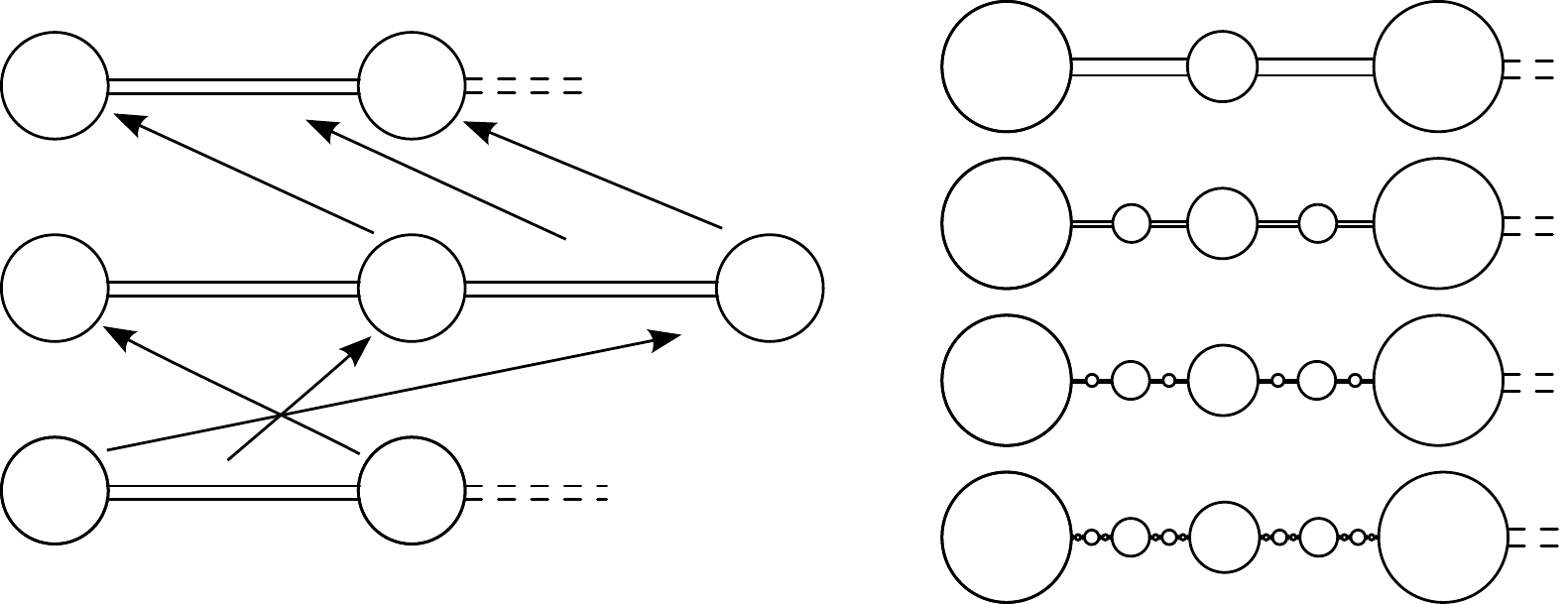}%
\caption{Continutation of Figure~\ref{fig:tent}. The spheres separating the left and middle sphere map $2:1$ to each other by $F_R$ or $F_{R^3}\circ F_R$, and repeating this, eventually to the middle sphere. At the limit these separating spheres could be thought of a countable collection filling the gaps in the complement of a Cantor set.}
\label{fig:aa}
\end{figure}

Up to now we observed the limits of the Julia sets in the three normalizations associated to the way the marked set degenerates.
However, we could look at many other normalizations.
For instance, we may focus on the tube linking the left sphere and the middle one (see Figure~\ref{fig:H2O}). Call it the $a$-tube, and call $b$-tube the one between the middle and right sphere. Since the left sphere is mapped to the right one and the middle to the left, the image of the tube has to span over the middle sphere. Recall that the map $F_R$ is $2:1$ on this tube. When the tube is long (i.e. $R$ is close to $1$), this map is close to the squaring map $z\mapsto z^2$ in appropriate coordinates. This is why we see these square rooted copies of the central sphere image: see Figure~\ref{fig:copies}. Moreover, this get replicated by futher preimages and explains the other rings visible in Figure~\ref{fig:copies}, which are more or less copies of each other by more and more square roots. Let us explain this (see also Figures~\ref{fig:tent} and~\ref{fig:aa}). For a given $R$ close to $1$, the $b$-tube in $\cal S_{R^{1/3}}$ must contain a nearly $1:1$ copy of the $a$-tube for $R$. The square root copy of the central sphere for $R$ found in the $a$-tube for $R^{1/3}$ is therefore also present in the $b$-tube for $R^{1/3^2}$. Now the image of the $a$-tube for $R^{1/3}$ not only spans over the middle sphere but also over the $a$-tube and $b$-tube for $R$: in fact the middle sphere represents a smaller and smaller part as $R\tend 1$. It follows that for each $R$, two fourth root copies of the central sphere must be found on the $a$ tube and two on the $b$ tube: these are roots of the central sphere for $R^{3^k}$ with different valuees of $k$; however all these powers of $R$ are close to $1$ since $R$ is. This goes on as on Figure~\ref{fig:aa} and accounts for all these rings we saw on the zoom.

\begin{figure}%
\includegraphics[width=11.5cm]{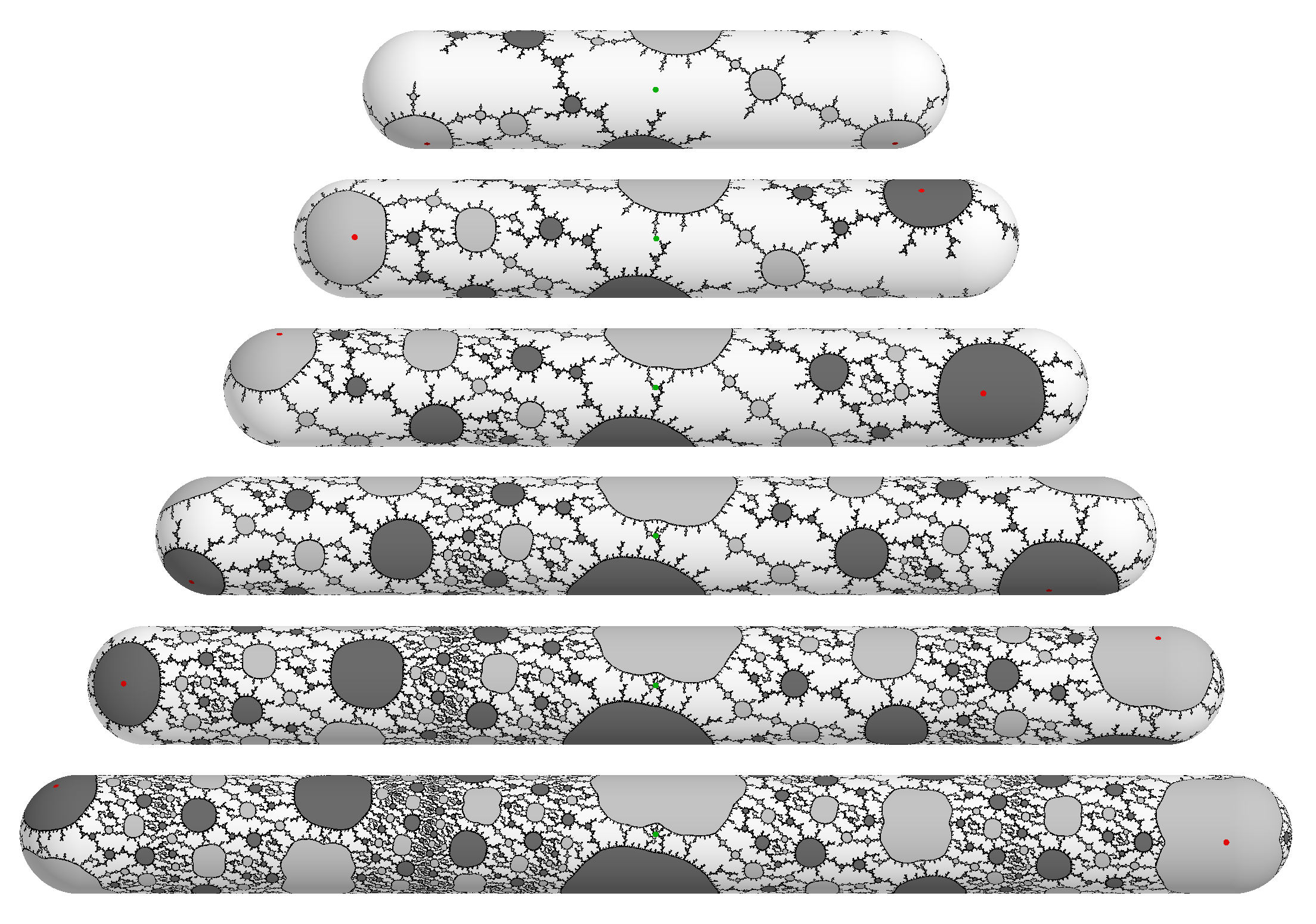}
\\
\includegraphics[width=11.5cm]{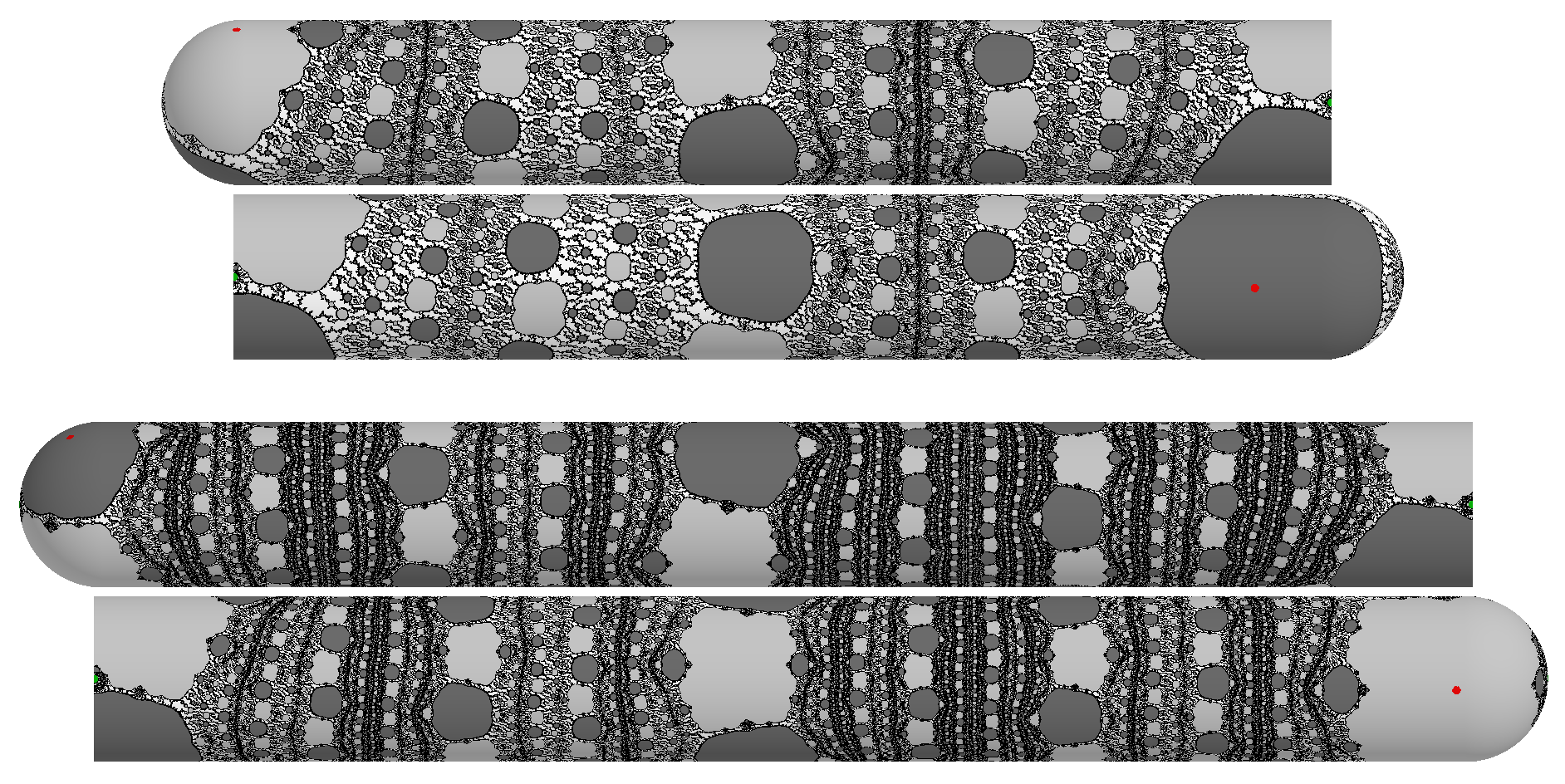}
\caption{Conformally correct pictures of $\cal S_R$ mapped to tubes, with $R=R_n=10^{4/3^n}$ and $n=5$, 6, 7, 8, 9, 10, 13 and 16.  The last two tubes were cut so as to fit in the frame. See the text for a more detailed description.}
\label{fig:tubes}
\end{figure}

On Figure~\ref{fig:tubes} we show another conformal representation of the Riemann surfaces $\cal S_R$. These are sticks consisting of a cylinder capped with two half spheres. The length of the stick increases with $n$ and is chosen so that we have limit pictures at the left, at the right and in the center, that are the limits of the picture in the three main normalizations as on Figure~\ref{fig:parallelMovie}. The chosen normalization here puts the points on the equator with angle $1/3,1/9,0$ respectively to the following points on the stick: the extremal left point, the central point facing the viewer, and the extremal right point. As $n$ increases, you will note that each end has an apparent rotation, while the central part has a non-rotating limit. Note that we also get (rotating) limits at $1/4$ and $3/4$ of the stick, and in fact at each $p/2^q$: these are the aforementioned rings of $2^q$-fold preimages of the central sphere. This stick representation makes the rings and their relationships more visible. As $R\tend 1$, the rings get further and further appart from each other.

If the tunnels on Figure~\ref{fig:tubes} have a length that grows at a linear pace, which seems to be the case, then the rings also separate at a linear pace. The closer the rings, the slower the pace. When the tube is conformally mapped to a geometric sphere, say according to the central sphere normalization, this linear growth becomes an exponential shrinking: each ring in the white basin has a size that shrinks to $0$ exponentially, with rings close to the boundary having a slower rate of convergence. More precisely for $R_n = e^{\textrm{const}/3^n}$, the ring at $p/2^q$ should become visible when $n \geq \on{const} q$ and its diameter should shrink like $e^{-\on{const}np/2^q}$ where $\textrm{const}$ designates different constants.
For each $q$ take the first $n$ when the $p/2^q$ are visible: this makes rings of diameter $e^{-\on{const}pq/2^q}$ for $p$ from $1$ to $q$.
In particular for all $\epsilon>0$, when $q$ gets big, there should be a lot of rings whose diameter is $>\epsilon$ and with $p$ odd, so that these rings are thin. This explains the bigger and bigger density of the Julia set in the white basin as $n$ grows.

\begin{table}
\setlength{\extrarowheight}{2pt}
\setlength{\arraycolsep}{0pt}
\begin{tabular}{|c|c|c|c|c|}
\hline
  $n$ & $R_n$ & $v_n$ & $v_{n}/v_{n-1}$ & $u_n/v_n$
\\ \hline
  2 & 2.78 & $\begin{array}{r}1.15\\ + i0.35\end{array}$ & $\begin{array}{r}0.0943 \\ + i0.0286\end{array}$ & $\begin{array}{r}-2.19 \\ + i0.48\end{array}$
\\ \hline
  3 & 1.40 & $\begin{array}{r}0.081\\ + i0.383\end{array}$ & $\begin{array}{r}0.158 \\ + i0.283\end{array}$ & $\begin{array}{r}-1.49 \\ + i1.33\end{array}$
\\ \hline
  4 & 1.120 & $\begin{array}{r} -0.113\\ + i0.064\end{array}$ & $\begin{array}{r}0.099\\ + i0.316\end{array}$ & $\begin{array}{r} -1.15\\ + i2.03\end{array}$
\\ \hline
  5 & 1.038 & $\begin{array}{r}-0.0324\\ - i0.0230\end{array}$ & $\begin{array}{r}0.129\\ + i0.276\end{array}$ & $\begin{array}{r}-1.202 \\ + i2.390\end{array}$
\\ \hline
  10 & $1\!+\!1.55e^{-4}$ & $\begin{array}{r}-1.16e^{-4}\\ + i0.55e^{-4}\end{array}$ & $\begin{array}{r}0.16016\\ + i0.27763\end{array}$ & $\begin{array}{r} -1.38658\\ + i2.40163\end{array}$
\\ \hline
  15 & $1\!+\!6.41e^{-7}$ & $\begin{array}{r}-0.36e^{-7}\\ + i4.35e^{-7}\end{array}$ & $\begin{array}{r}0.1602506\\ + i0.2775626\end{array}$ & $\begin{array}{r} -1.386716\\ + i2.401863\end{array}$
\\ \hline
  20 & $1\!+\!2.64e^{-9}$ & $\begin{array}{r}1.21e^{-9}\\ + i0.84e^{-9}\end{array}$ & $\begin{array}{r}0.160249937\\ + i0.277561021\end{array}$ & $\begin{array}{r} -1.38672276\\ + i 2.40187419\end{array}$
\\ \hline
  25 & $1\!+\!1.08e^{-11}$ & $\begin{array}{r}4.51e^{-12}\\ - i2.12e^{-12}\end{array}$ & $\begin{array}{r}0.1602499527\\ + i0.2775610603\end{array}$ & $\begin{array}{r}-1.386722542 \\ + i2.401873901\end{array}$
\\ \hline
 30 & $1+4.47e^{-14}$ & $\begin{array}{r}0.14e^{-15}\\ - i1.68e^{-14}\end{array}$ & $\begin{array}{r}0.1602499522\\ + i0.2775610592\end{array}$ & $\begin{array}{r}-1.386722548 \\ + i2.401873910\end{array}$
\\ \hline
\end{tabular}
\caption{Computer experiment. Numbers rounded by truncation.}
\label{tab:exp}
\end{table}

\begin{table}
\setlength{\extrarowheight}{2pt}
\setlength{\arraycolsep}{0pt}
\begin{tabular}{|c|c|}
\hline
  $n$ & $u_n/v_{n-1}$
\\ \hline
  2 & $\begin{array}{r}-0.221  - i0.017\end{array}$
\\ \hline
  3 & $\begin{array}{r}-0.615  - i0.212\end{array}$
\\ \hline
  4 & $\begin{array}{r}-0.759  - i0.162\end{array}$
\\ \hline
  5 & $\begin{array}{r}-0.817  - i0.023\end{array}$
\\ \hline
  10 & $\begin{array}{r}-0.88886361  - i0.00029954\end{array}$
\\ \hline
  15 & $\begin{array}{r}-0.88888972  - i0.00000057\end{array}$
\\ \hline
  20 & $\begin{array}{r}-0.88888889199  - i1.46e^{-9}\end{array}$
\\ \hline
  25 & $\begin{array}{r}-0.88888888888985  - i1.158e^{-11}\end{array}$
\\ \hline
 30 & $\begin{array}{r}-0.8888888888888565  - i2.24e^{-14}\end{array}$
\\ \hline
\end{tabular}
\caption{Computer experiment. Numbers rounded by truncation.}
\label{tab:exp2}
\end{table}

\begin{figure}%
\begin{tikzpicture}
\node at (0,0) {\includegraphics[width=7cm]{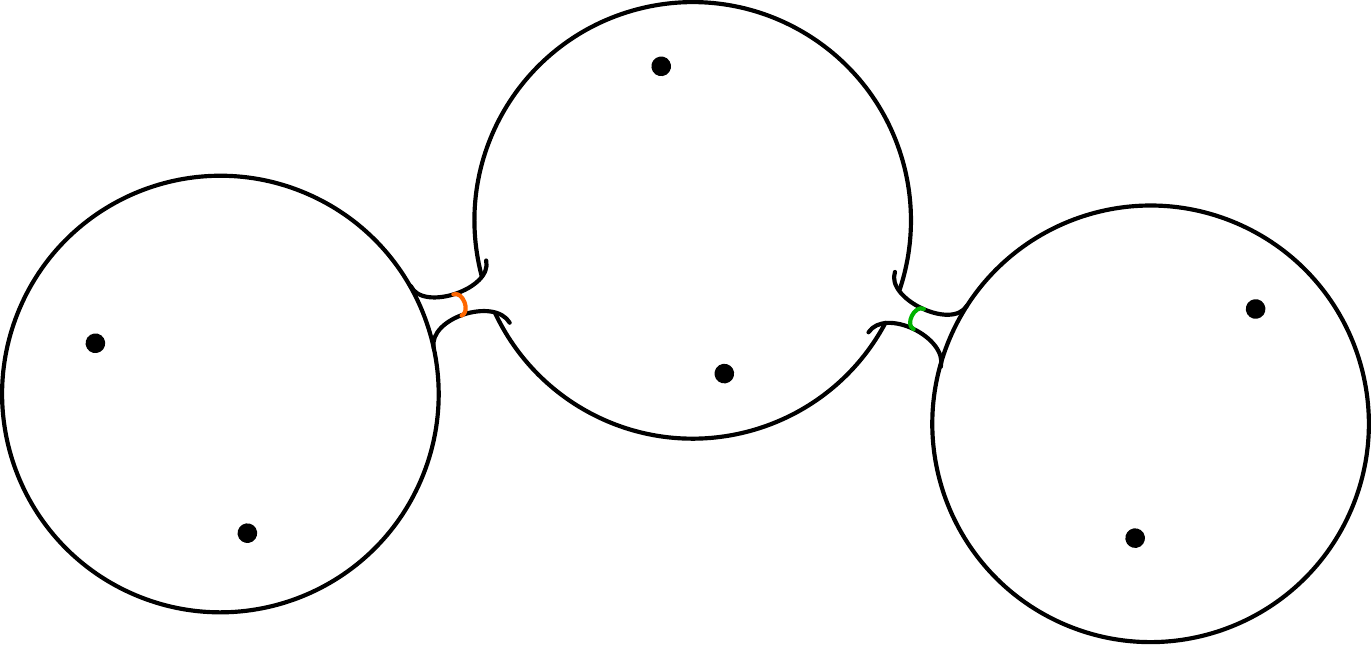}};
\node at (-0.1,-0.3) {$c_0$};
\node at (-1.9,-1) {$c_1$};
\node at (2,-1) {$c_2$};
\node at (-0.3,1.1) {$x$};
\node at (-2.8,0) {$y$};
\node at (2.7,0) {$y_1$};
\node at (-1.1,-0.3) {$a$};
\node at (1.1,-.3) {$b$};
\node at (-2.3,1.1) {$A_1$};
\node at (0,-1) {$A_2$};
\node at (2.4,1) {$A_3$};
\end{tikzpicture}
\caption{Reminder of Figure~\ref{fig:H2O}.}
\end{figure}

To finish, we would like to show a computer experiment measuring the rate at which the marked points group together, which is related to the rate of growth of the tubes, and the rate of shrinking of the geodesics homotopic to the canonical obstruction.
Let us choose our central sphere normalization, where the points labeled $c_0,e_0,x$ are mapped respectively to $0,1,\infty$.
Let $R_n = 10^{4/3^n}$ and $v_n$ be difference of complex coordinates from $y$ to $c_1$ (in the chosen normalization) and $u_n$ from $y_1$ to $c_2$.
On Table~\ref{tab:exp} we list, as a function of $n$ with $R=R_n$, the value of $v_n$, of the rate $v_{n}/v_{n-1}$ and of the ratios $u_n/v_n$.
The ratio $u_n/v_n$ seems to converge: the two bunches of points converge at the same rate, which is coherent with the observation that the middle sphere is near the middle of the tubes on Figure~\ref{fig:tubes} and coherent too with the fact that the tunnel in $\cal S_R$ between the middle and right sphere is a $1:1$ pre image by $F_R$ of the other tunnel for $\cal S_{R^3}$ (growing linearly). The latter remark suggests to look at the limit of the ratio $u_n/v_{n-1}$ and see if it is close to have modulus one: indeed, according to Table~\ref{tab:exp2}, it seems to converge to $-8/9$.
The rate $v_{n}/v_{n-1}$ seems to tend to a complex number
\[\rho\approx0.1602499522+i0.277561059.\]
Putting the square of its norm in Simon Plouffe's inverter\footnote{Available on the Internet at \texttt{http://pi.lacim.uqam.ca/} as of January 2012.} suggests it might be 
\[\rho \overset{\scriptstyle ?}= 2e^{i\pi/3}/3^{5/3}.\]
It shall not come as a surprise that the limit rate of convergence be an algebraic number: it is probably the leading eigenvalue of a well-defined branch of the pull-back map at some fixed point on some stratum of the pinched sphere compatification of the moduli space (or a compactification of the Teichmüller space?). See \cite{Se}.

\begin{question} Compute this eigenvalue and prove it is equal to $\rho$.
\end{question}

Even though from Table~\ref{tab:exp}, $R_n-1$ and $v_n$ seem to decrease to $0$ at the same rate, we expect one to tend slightly faster: $R_n -1 \sim \ln(10)/3^n$ whereas $v_n \sim \on{constant}\rho^n$, and $|\rho|=0.320\ldots \neq 1/3 = 0.333\ldots$

\bibliographystyle{alpha}
\bibliography{forArXiv}

\end{document}